\documentclass{amsart}

\usepackage{amsthm,amsmath,amssymb,amsfonts}
\usepackage[colorlinks=true,hypertexnames=false,linkcolor=blue,citecolor=blue]{hyperref}
\usepackage{epsfig,graphics}
\usepackage{amscd}
\usepackage{amsmath}
\usepackage{bm}
\usepackage{amssymb}
\usepackage{graphicx}
\usepackage{verbatim}
\usepackage[percent]{overpic}
\usepackage{tikz}
\usepackage{hyperref}
\makeatletter
\newcommand{\labitem}[2]{
\def\@itemlabel{\textbf{#1}}
\item
\def\@currentlabel{#1}\label{#2}}
\makeatother

\newtheorem{theorem}{Theorem}[section]
\newtheorem{lemma}[theorem]{Lemma}
\newtheorem{claim}[theorem]{Claim}
\newtheorem{proposition}[theorem]{Proposition}
\newtheorem{corollary}[theorem]{Corollary}

\theoremstyle{definition}
\newtheorem{definition}[theorem]{Definition}
\newtheorem{example}[theorem]{Example}

\theoremstyle{remark}
\newtheorem{remark}[theorem]{Remark}
\newtheorem{remarks}[theorem]{Remarks}

\numberwithin{equation}{section}

\theoremstyle{plain}
\newtheorem{maintheorem}{Theorem}

\newcommand{\R}{\ensuremath{\mathbb{R}}}

\newcommand{\nt}{\ensuremath{\mathbb{N}}}
\newcommand{\II}{\ensuremath{\mathbb{I}}}

\newcommand{\interior}{\operatorname{Int}}
\newcommand{\closure}{\operatorname{Cl}}
\newcommand{\boundary}{\operatorname{Bd}}
\newcommand{\ddi}{\operatorname{d\!}}
\newcommand{\dd}{\operatorname{d}}
\newcommand{\Di}{\operatorname{D}}
\newcommand{\D}{\operatorname{D}\!}
\newcommand{\res}[2]{{#1}\raisebox{-1ex}{$\left|\,_{#2}\right.$}}
\newcommand{\s}{\operatorname{S}\!}
\newcommand{\distance}{\operatorname{dist}}
\newcommand{\Bimod}{\operatorname{Bimod}}
\newcommand{\cir}{\mathbb{T}^1}

\begin{document}
\textcolor{black}{}\global\long\def\TDD#1{{\color{red}To\, Do(#1)}}

\title{Ergodic properties of bimodal circle maps}

\author{Sylvain Crovisier}
\address{Laboratoire de Math\'ematiques d'Orsay, CNRS - Universit\'e Paris-Sud}
\curraddr{B\^atiment 425, F-91405 Orsay Cedex, France}
\email{Sylvain.Crovisier@math.u-psud.fr}

\author{Pablo Guarino}
\address{Instituto de Matem\'atica e Estat\'istica, Universidade Federal Fluminense}
\curraddr{Rua M\'ario Santos Braga S/N, 24020-140, Niter\'oi, Rio de Janeiro RJ, Brazil}
\email{pablo\_\,guarino@id.uff.br}

\author{Liviana Palmisano}
\address{IMPAN, Institute of Mathematics, Polish Academy of Sciences, Warsaw}
\email{liviana.palmisano@gmail.com}

\subjclass[2010]{Primary }

\keywords{}

\begin{abstract} We give conditions that characterize the existence of an absolutely continuous invariant probability measure for a degree one $C^2$ endomorphism of the circle which is bimodal, such that all its periodic orbits are repelling, and such that both boundaries of its rotation interval are irrational numbers. Those conditions are satisfied when the boundary points of the rotation interval belong to a Diophantine class. In particular they hold for Lebesgue almost every rotation interval. By standard results, the measure obtained is a global physical measure and it is hyperbolic.
\end{abstract}

\maketitle

\section{Introduction}

Let $f$ be a $C^r$ map of a compact interval (or the unit circle) to itself, for some $r \geq 1$. Such a map is called \emph{uniformly hyperbolic} (or \emph{Axiom A}) if it has a finite number of hyperbolic periodic attractors\footnote{A periodic point $p$ of period $n$ is called a \emph{hyperbolic attractor} if $\big|Df^n(p)\big|\in(0,1)$.} and the complement of the union of its basins of attraction, usually denoted by $\Sigma(f)$, is expanding\footnote{We say that $\Sigma(f)$ is \emph{expanding} under $f$ if there exist two constants $C>0$ and $\alpha>1$ such that $\big|Df^n(x)\big|>C\alpha^n$ for all $x\in\Sigma(f)$ and $n\in\nt$.}. The most tame examples of uniformly hyperbolic maps are the so-called \emph{Morse-Smale} diffeomorphisms, where $\Sigma(f)$ is just a finite union of repelling periodic orbits. On the other hand, $\Sigma(f)$ may be a Cantor set or even the whole phase-space (take, for instance, $z \mapsto z^n$ in the unit circle with $n \geq 2$).

Any Axiom A map for which all its critical points (if they exist) are non-degenerate and have disjoint orbits is \emph{structurally stable}: any $C^r$-nearby map is conjugate to it \cite[Section III.2, Theorem 2.5]{demelovanstrien}.

A major result in \emph{real} one-dimensional dynamics states that uniformly hyperbolic dynamics are (open and) dense in the space of $C^r$ maps of any given compact interval into itself, and for any given $r=1,2,...,\infty,\omega$ (see \cite{densidad} and the references therein). Actually even more is true: any real polynomial can be approximated by hyperbolic real polynomials of the same degree \cite[Theorem 1]{densidad}.

From the \emph{topological} viewpoint, therefore, \emph{most} one-dimensional dynamical systems are uniformly hyperbolic. By considering families parametrized by \emph{finite dimensional} manifolds, one can ask about most dynamical systems from a \emph{probabilistic} viewpoint, with respect to Lebesgue measure on parameter space (see also the recent \emph{global} probabilistic approach for circle diffeomorphisms considered in \cite{triestino}). With this purpose, we say that an interval map is \emph{stochastic} if it admits an invariant Borel probability measure which is absolutely continuous with respect to the Lebesgue measure.

It is not difficult to prove that if $f$ has critical points, is $C^{1+\alpha}$ and Axiom A, then its corresponding set $\Sigma(f)$ has zero Lebesgue measure (\cite[Section III.2, Theorem 2.6]{demelovanstrien}, see also Proposition \ref{frmap} in this paper). In this case, the support of any invariant Borel probability measure has zero Lebesgue measure and in particular no invariant measure of an Axiom A map with critical points can be absolutely continuous with respect to Lebesgue.

Let us illustrate this discussion with a classical example: for each $t\in(0,4)$ consider the quadratic polynomial $f_{t}:[0,1]\to[0,1]$ given by $f_{t}(x)=t\,x\,(1-x)$. This one-parameter family, the so-called \emph{quadratic family}, was introduced by R. May in 1976 \cite{may} as a model for the growth, or fluctuation, of biological populations. As explained above, there exists an open and dense set of parameters $t\in(0,4)$ such that the corresponding polynomials $f_{t}$ are Axiom A. For those maps, Lebesgue almost every point converges to a unique attracting periodic orbit (this includes the critical point).

However, in the early eighties, Jakobson \cite{jakobson} proved the existence of a \emph{positive} measure set of parameters $t\in(0,4)$ such that the corresponding maps are stochastic. Therefore, at least from the probabilistic viewpoint, stochastic dynamics are not negligible. Later, Lyubich proved that $f_{t}$ is either Axiom A or stochastic for a \emph{full} Lebesgue measure set of parameters $t\in(0,4)$ \cite{Lydicotomia}. Under suitable conditions, the same dichotomy holds for \emph{generic} one-parameter families of real analytic unimodal maps \cite{arturlyubichwelington}.

After Jakobson result, \emph{metric} conditions were shown to be sufficient for a smooth interval map to be stochastic (many papers have addressed this problem, see \cite{colleteckmann}, \cite{bencar}, \cite{bowen}, \cite{nowvs}, \cite{martensnow}, \cite{acipmultimodal} and also \cite[Chapter V]{demelovanstrien} and the references therein). Those conditions are usually related to the growth of the derivative along the critical orbit (see Remark \ref{remlargeder} below). In this paper we look for more \emph{combinatorial} conditions, instead of metric ones (see Theorem \ref{thmA} and Theorem \ref{thmB} below). With our approach we are able to give not only sufficient but also necessary conditions (see conditions \ref{C-} and \ref{C+} in Section \ref{statements} below) for a bimodal degree one circle endomorphism, such that all its periodic orbits are repelling and with irrational combinatorics of both critical points, to be stochastic (Theorem \ref{thmA}). Moreover, we provide a big class of maps satisfying those conditions (Theorem \ref{thmB}).

Before to explain our results formally we briefly review some basic definitions and statements. We refer to the book of de Melo and van Strien \cite{demelovanstrien} for general background in one-dimensional dynamics.

\subsection{Our setting}\label{setting} Let $\mathbb{T}^1$ be the circle and $\pi:\mathbb{R}\rightarrow\mathbb{T}^1$ its universal covering. The Lebesgue measure on $\mathbb{R}$ and $\mathbb{T}^1$ (the Haar measure) will be denoted by $\lambda$, and the usual distance on the circle by $\dd$. A map $f:\mathbb{T}^1\rightarrow\mathbb{T}^1$ is an \emph{endomorphism} of the circle if there exists a lift $\tilde f:\mathbb{R}\rightarrow\mathbb{R}$ which is continuous and satisfies: 
\begin{itemize}
\item $\pi\circ\tilde f=f\circ \pi,$
\item for any $x\in \mathbb{R}$, $\tilde f(x+1)=\tilde f(x)+1$.
\end{itemize}

To $\tilde f$ is associated a rotation set $\mathcal{R}(\tilde f)$, which is a compact subinterval $[\rho^-,\rho^+]$ of $\mathbb{R}$ (see~\cite{chencinergt,misiurewicz}).

\begin{definition}\label{ourspace} Let $\Bimod$ be the set of endomorphisms $f$ such that for some lift $\tilde f$ with rotation interval $[\rho^-,\rho^+]$ the following three properties are satisfied.

\begin{description}
\item[(A1)] The map $f$ is $C^2$ and bimodal: there exist $0<\tilde c^+<\tilde c^-<1$ such that the restrictions of $\tilde f$ on $(\tilde c^+,\tilde c^-)$ and $(\tilde c^-,\tilde c^++1)$ are respectively decreasing and increasing diffeomorphisms onto their image. The critical points $c^+=\pi(\tilde c^+)$ and 
$c^-=\pi(\tilde c^-)$ are non-flat: there exist some constant $\ell^-> 1$ (resp. $\ell^+> 1$) and some $C^2$-diffeomorphism $\psi^-$ (resp. $\psi^+$) from $\mathbb{R}$ into itself such that $\psi^-(c^-)=\psi^+(c^+)=0$ and such that near $\tilde c^-$ (resp. $\tilde c^+$),$$\tilde f=\tilde f(\tilde c^-)+|\psi^-|^{\ell^-} \text{ (resp. }\tilde f=\tilde f(\tilde c^+)-|\psi^+|^{\ell^+}).$$
\item[(A2)] The rotation numbers $\rho^-$ and $\rho^+$ are irrational.
\item[(A3)] All periodic orbits of $f$ are hyperbolic repelling:$$\forall\,x\in \mathbb{T}^1,\;\forall\,n\geq 1,\;f^n(x)=x\Rightarrow |\D f^n(x)|>1.$$
\end{description}
\end{definition}

One specifies also the subsets $\Bimod(\ell^-,\ell^+)$ when the constants $\ell^-,\ell^+$ that appear in (A1) have to be fixed. Let us remark also that both $\ell^-> 1$ and $\ell^+>1$ are real numbers, not necessarily integers. 

\begin{figure}[h]
\begin{center}
\begin{overpic}[scale=1.5]{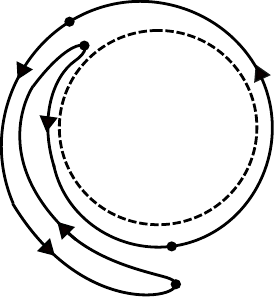}
\put(75, 60){$\mathbb{T}^1$}
\put(95, 60){$\mathbb{T}^1$}
\put(60, 10){$d^+$}
\put(60, -3){$c^+$}
\put(17, 93){$d^-$}
\put(30.5, 85.5){$c^-$}
\end{overpic}
\caption{Maps in $\Bimod$ are degree one branched coverings of the circle. They all have rich topological dynamics: they present periodic orbits of arbitrarily large period, they are topologically mixing in the whole circle (see Proposition \ref{topmix} in Appendix \ref{apptopasp}), they have positive topological entropy and they exhibit sensitive dependence to initial conditions.}
\label{figcovering}
\end{center}
\end{figure}
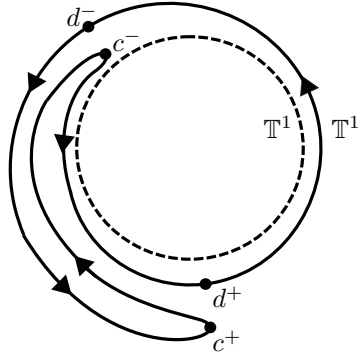

%\begin{remark} Maps in $\Bimod$ have rich topological dynamics: they present periodic orbits of arbitrarily large period, they are topologically mixing in the whole circle (see Proposition \ref{topmix} in Appendix \ref{apptopasp}), they have positive topological entropy and they exhibit sensitive dependence to initial conditions.
%\end{remark}

\begin{example} We give now some examples of maps belonging to the class $\Bimod$. Let $f$ be a bimodal map satisfying condition (A1) above. Suppose that $f$ is $C^3$ and that its Schwarzian derivative$$\s f=\frac{\D^3 f}{\D f}-\frac{3}{2}\left(\frac{\D^2 f}{\D f}\right)^2$$is strictly negative on $\mathbb{T}^1\setminus \{c^+,c^-\}$. A classical result of Singer (see \cite[Section II.6, Theorem 6.1]{demelovanstrien}) implies in this case that any non-repelling periodic orbit of $f$ has to be topologically attracting and, moreover, its immediate basin of attraction contains a critical point of $f$. However, if $\rho^-$ and $\rho^+$ are irrational numbers, the critical orbits cannot accumulate on any periodic orbit. Therefore the negative Schwarzian condition combined with conditions (A1) and (A2) imply condition (A3) in the $C^3$ category. For example, condition (A1) and the negative Schwarzian condition hold for the Arnol'd family (see \cite{arnold}):$$\tilde f_{a,\omega}(x)=x+a\sin(2\pi x)+\omega,\quad (a>1/2\pi,\; \omega\in\mathbb{R}).$$

In this case both critical points are non-degenerate ($\ell^-=\ell^+=2$). Moreover, any compact interval $[\rho^-,\rho^+]$ with non-empty interior is realized as the rotation interval of $\tilde f_{a,\omega}$ for some parameter $(a,\omega)$ (and $(a,\omega)$ is unique if $\rho^+-\rho^->2$ and $\rho^+,\rho^-\not \in\mathbb{Q}$, see~\cite{these}, Section 2.1).
\end{example}

\subsection{Basic definitions} The aim of this paper is to study invariant measures for functions $f\in\Bimod$.

%\begin{definition} A borelian probability measure $\mu$ is said to be \emph{absolutely continuous} with respect to the Lebesgue measure $\lambda$ (resp. \emph{equivalent} to $\lambda$) if for any borelian $A\subset \mathbb{R}$,$$\lambda(A)=0\Rightarrow \mu(A)=0,\quad(\text{resp. }\lambda(A)=0\Leftrightarrow \mu(A)=0).$$
%\end{definition}

\begin{definition}\label{deffis} An $f$-invariant borelian probability measure $\mu$ is a \emph{global physical measure} for $f$ if for Lebesgue almost every $x\in\cir$ we have:$$\lim_{n\to+\infty}\left\{\frac{1}{n}\sum_{j=0}^{n-1}\phi\big(f^j(x)\big)\right\}=\int_{\cir}\phi\,d\mu\quad\mbox{for any continuous function $\phi:\cir\to\R$.}$$
\end{definition} 

\begin{definition}\label{defhypmea} An ergodic $f$-invariant borelian probability measure $\mu$ is \emph{hyperbolic} if its Lyapunov exponent $\int_{\cir}\log |D f| \ddi \mu$ is strictly positive.
\end{definition}

\begin{definition}\label{dioph} A real number $\rho$ is \emph{diophantine} with exponent $\beta>0$ if there exists a constant $C>0$ such that for any rational number $\frac{p}{q}$,$$|q\rho-p|\geq C q^{-(1+\beta)}.$$
\end{definition}

\subsection{Statements of the main results}\label{statements} We fix a function $f\in\Bimod$ with rotation interval $[\rho^-,\rho^+]$. Let us consider the rational approximations $(\frac{p^-_k}{q^-_k})$ and $(\frac{p^+_k}{q^+_k})$ of $\rho^-$ and $\rho^+$ given by the continued fraction expansions. We will only use $q_k^-$ and $q_k^+$ and their precise definition will be recalled at Section~\ref{rotation}. However it is important here to note that with our definitions, for any $k\in\mathbb{N}$,$$\frac{p^-_{2k}}{q^-_{2k}}<\rho^-<\frac{p^-_{2k+1}}{q^-_{2k+1}},\quad\text{and}\quad\frac{p^+_{2k}}{q^+_{2k}}<\rho^+<\frac{p^+_{2k+1}}{q^+_{2k+1}}.$$

We introduce two conditions on $f$:
\begin{description}
\labitem{(C$^-$)}{C-} The series $\displaystyle\sum_{k\geq 0}q^-_{2k+1}\dd\big(f^{q^-_{2k}}(c^-),c^-\big)$ is finite.
\labitem{(C$^+$)}{C+} The series $\displaystyle\sum_{k\geq 1}q^+_{2k}\dd\big(f^{q^+_{2k-1}}(c^+),c^+\big)$ is finite.
\end{description}

The aim of this paper is to show that conditions \ref{C-} and \ref{C+} characterize stochastic dynamics in the class $\Bimod$. Indeed, our first main result is the following:

\begin{maintheorem}\label{thmA} An endomorphism $f\in\Bimod$ preserves a probability measure $\mu$ which is absolutely continuous with respect to the Lebesgue measure $\lambda$ if and only if conditions \ref{C-} and \ref{C+} are both satisfied.

When such an absolutely continuous measure $\mu$ exists, it is unique, equivalent to $\lambda$ and ergodic. In particular $\mu$ is a global physical measure for $f$. Moreover, $\mu$ is hyperbolic and has positive metric entropy.
\end{maintheorem}

From the proof of Theorem \ref{thmA} it will be clear however (see Section~\ref{pfs1}) that any map in $\Bimod$ preserves a $\sigma$-finite measure which is equivalent to the Lebesgue measure $\lambda$. Conditions \ref{C-} and \ref{C+} assert that this measure is finite.

As already mentioned in the abstract, the facts that the measure $\mu$ is a global physical measure and it is hyperbolic, follow at once after existence of $\mu$ is established (see Section \ref{pfs1} for more details and the corresponding references).

Note, finally, that both conditions \ref{C-} and \ref{C+} are \emph{quantitative}. In general it is not possible to give a \emph{topological} condition equivalent to the existence of an absolutely continuous invariant probability measure. Indeed, as Arnol'd showed in the early sixties (\cite{arnold}, see also \cite[Section I.5]{demelovanstrien}), there exist real analytic circle diffeomorphisms topologically conjugate to a rotation, which do not preserve any absolutely continuous invariant measure (see \cite{bruin} for examples in the class of unimodal maps).

J.~Graczyk has given in~\cite{jacekmanuscript} precise estimates on the distances $\dd(f^{q^-_{2k}}(c^-),c^-)$ and $\dd(f^{q^+_{2k+1}}(c^+),c^+)$. We will show that, together with Theorem \ref{thmA}, Graczyk's estimates imply that any bimodal endomorphism whose rotation interval satisfies a diophantine condition preserves an absolutely continuous probability measure. More precisely, our second main result is the following:

\begin{maintheorem}\label{thmB} For any constants $\ell^-,\ell^+>1$ there exists $\beta=\beta(\ell^-,\ell^+)>0$ with the following property: if $f\in\Bimod(\ell^-,\ell^+)$ is an endomorphism with rotation interval $[\rho^-,\rho^+]$ such that both $\rho^-$ and $\rho^+$ are Diophantine with exponent $\beta$, then $f$ preserves a probability measure which is absolutely continuous with respect to $\lambda$.

On the other hand, there exist Liouvillian numbers $\rho^-$ or $\rho^+$ such that if $f\in\Bimod(\ell^-,\ell^+)$ is an endomorphism with rotation interval $[\rho^-,\rho^+]$, then $f$ does not preserve any probability measure which is absolutely continuous with respect to $\lambda$.
\end{maintheorem}

We do not give an optimal arithmetic condition on $\rho^-$ and $\rho^+$. Note that the finer descriptions one could get are not symmetric with respect to the coefficients of the continued fraction representations of $\rho^-$ and $\rho^+$.

Theorem \ref{thmA} and Theorem \ref{thmB} describe the dynamics of bimodal circle endomorphisms for almost any rotation interval (recall that the set of Diophantine rotation numbers with exponent $\beta$ has full Lebesgue measure in $[0,1]$ for any $\beta>0$). However, \'Swi\c{a}tek has proved in~\cite{swiatek} that in the Arnol'd-like families the corresponding parameters $(a,\omega)$ have zero Lebesgue measure in $\mathbb{R}^2$.

In some way, our results can be compared to linearization theorems: for any smooth enough diffeomorphism of the circle with diophantine rotation number, M.~Herman proved in~\cite{hermanthesis} that the conjugacy $h$ to the rotation is a diffeomorphism. By pulling back the Lebesgue measure by $h$ one gets also an invariant probability measure which is equivalent to the Lebesgue measure.

\begin{remark}\label{remlargeder} As already mentioned, one of the first results that showed existence of absolutely continuous invariant measure for smooth one-dimensional maps with some recurrent critical point is certainly Jakobson's theorem in~\cite{jakobson}. In any proof one needs to avoid strong recurrence of the critical orbits near the critical points. This control is obtained here thanks to the combinatorics of the rotations with angles $\rho^-$ and $\rho^+$ which describe the forward orbits of $c^-$ and $c^+$ respectively.

More precisely, in \cite{acipmultimodal} it has been proved that if$$\lim_{n\to+\infty}\Big|Df^n\big(f(c^+)\big)\Big|=+\infty\quad\mbox{and}\quad\lim_{n\to+\infty}\Big|Df^n\big(f(c^-)\big)\Big|=+\infty\,,$$then $f\in\Bimod$ admits an absolutely continuous invariant probability measure. We do not know whether those \emph{large derivatives} conditions follow from conditions \ref{C-} and \ref{C+} in the $\Bimod$ class (recall that the existence of an absolutely continuous invariant probability measure does not imply positive Lyapunov exponent at the critical value, see \cite{Lyu-Mil}).
\end{remark}

The proof of Theorem \ref{thmA} relies on a classical method of \emph{inducing}, developed during the sixties and seventies by Adler, Weiss, Bowen, Jakobson and Sinai among others (see \cite[Section V.3]{demelovanstrien} and the references therein): given $f\in\Bimod$ we will consider the maximal closed interval $\II$ where $f$ is decreasing (see Figure \ref{fig1bimod}). We will prove in Section \ref{markovprop} that Lebesgue almost every point of $\mathbb{T}^1$ enters $\II$ under the action of $f$, and that this \emph{first entry map} is a Markov map (see Definition \ref{markovdef} in Section \ref{markovprop}). A classical result in one-dimensional dynamics (the \emph{folklore} theorem) assures that this Markov map preserves a probability measure which is equivalent to Lebesgue, ergodic and hyperbolic (see Theorem \ref{induce}). This Markovian structure holds for any map in $\Bimod$. We will prove in Section~\ref{mainthm} that one can \emph{lift} this Markov measure to a \emph{finite} invariant measure for $f$ if and only if conditions \ref{C-} and \ref{C+} hold.

The proof of Theorem \ref{thmB} is given in Section \ref{provadoB}, and relies on precise estimates on the distances $\dd\big(f^{q^-_{2k}}(c^-),c^-\big)$ and $\dd\big(f^{q^+_{2k-1}}(c^+),c^+\big)$ obtained by Graczyk in \cite{jacekmanuscript}, see Theorem \ref{thmgraczyk}.

\subsection{Organization of the paper} Basic constructions, combinatorics of rotation, upper maps and extended upper maps are described in Section~\ref{preliminaries}. Section~\ref{induces} contains the definition of first return maps to different intervals and Section~\ref{distprop} its distortion properties. In Section~\ref{markovprop}  we prove that the first return map to the interval where the function is decreasing is a Markov map and that conditions \ref{C-} and \ref{C+} imply that the return time is summable (Section~\ref{sumcond}). We prove the main theorems in Section~\ref{mainthm}.

\textbf{Ackowledgements} The first author would like to thanks J.~Graczyk, D.~Sands and J.-C. Yoccoz for their advices and their support. The third author was supported by funds allocated to the implementation of the international co-funded project in the years 2014-2018, 3038/7.PR/2014/2, and by the EU grant PCOFUND-GA-2012-600415. During the preparation of this article, S.C. visited IMPAN and L.P. visited PUC-Rio. We wish to thank both institutions for their warm hospitality.

\section{Preliminaries}\label{preliminaries}

\subsection{Notations and Definitions} 
We introduce some basic notation used along this paper.
\begin{itemize}
\item For any topological space $X$, the interior, closure and boundary of a subset $Y\subset X$ will be denoted by $\interior(Y)$, $\closure(Y)$ and $\boundary(Y)$.
\item The integer part of a real number $x$ is $[x]$ (thus $[x]\leq x<[x]+1$).
\item The rotation by $\rho$ on $\mathbb{T}^1$ is denoted by $R_\rho$.
\item 
 Let $x$ and $y$ be two points on $\mathbb{T}^1$. One can find some lifts $\tilde x$ and $\tilde y$ in $\mathbb{R}$ with $\tilde x\leq \tilde y\leq \tilde x+1$. Then, one defines the interval $[x,y]$ as the interval 
 $\pi([\tilde x,\tilde y])$. It does not depends on the choice of the lifts $\tilde x$ and $\tilde y$. In the same way, one defines the intervals $(x,y)$, $[x,y)$ and $(x,y]$.
\end{itemize}
 
\begin{definition}\label{premdef}
\renewcommand{\theenumi}{\roman{enumi}}
\begin{enumerate}
\item An \textbf{interval} $I$ of $\mathbb{T}^1$ is a connected subset of $\mathbb{T}^1$. We denote by $|I|$ its length, i.e. its measure with respect to the Lebesgue measure $\lambda$.
\item Two intervals $I,I'$ of $\mathbb{T}^1$ are \textbf{adjacent} if$$\interior(I)\cap\interior(I')=\emptyset,\text{ and }\closure(I\cup I') \text{ is an interval.}$$
\item Let $\mathcal{N}$ be a family of intervals of $\mathbb{T}^1$. A map $N:\mathcal{N}\rightarrow\mathbb{N}$ associated to this family is \textbf{summable} if the following quantity is finite:$$\sum_{J\in\mathcal{N}}N(J)|J|.$$
\item A family $\mathcal{N}$ of intervals of $\mathbb{T}^1$ is a \textbf{measurable partition} of an interval $T\subset \mathbb{T}^1$ if:
\begin{itemize}
\item for any $J\in\mathcal{N}$, $\interior(J)\not=\emptyset$, $\interior(J)\subset T$ and for any $J,J'\in\mathcal{N}$,$$J\cap J'\not =\emptyset\Rightarrow J=J';$$
\item $\lambda\left(T\setminus \bigcup_{J\in\mathcal{N}} J\right)=0$.
\end{itemize}
\item For any two disjoint sets $X$ and $Y$ contained in $\mathbb{R}$ (or contained in a same interval $I\varsubsetneq\mathbb{T}^1$ that has to be specified), the notation $X<Y$ will stand for:$$\forall x\in X, \; \forall y\in Y,\quad x<y.$$
\item On any interval $I$ of $\mathbb{T}^1$ which is not the full circle $\mathbb{T}^1$, one defines an order in the following way: one chooses any lift $\tilde I$, i.e. an interval in $\mathbb{R}$ such that $\pi:\tilde I\rightarrow I$ is an homeomorphism. The order on $I$ is obtained by identification with the usual order on $\tilde I$. It does not depends on the choice of the lift $\tilde I$.
\end{enumerate}
\end{definition}

\subsection{Upper maps}\label{extenint}  Let $f\in\Bimod$. We refer to the notation introduced in Subsection \ref{setting} and we set:

\begin{itemize}
\item $\tilde \II=[\tilde c^+,\tilde c^-]$ and $\II=\pi(\tilde \II)$ the maximal closed interval where $f$ is decreasing,
\item ${\tilde \II}^+=[\tilde c^+,\tilde d^+]$ and $\II^+=\pi(\tilde \II^+) $, and
\item $\tilde \II^-=[\tilde d^-,\tilde c^-]$ and $\II^-=\pi(\tilde \II^-)$,
\end{itemize}
where $\tilde d^+\in(\tilde c^-,\tilde c^++1)$ is defined by $\tilde f(\tilde d^+)=\tilde f(\tilde c^+)$ (see Figures \ref{figcovering} and \ref{fig1bimod}).

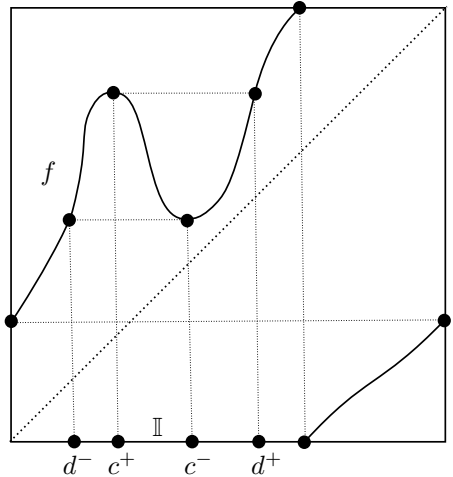
\begin{figure}[h]
\begin{center}
\begin{overpic}[scale=0.4]{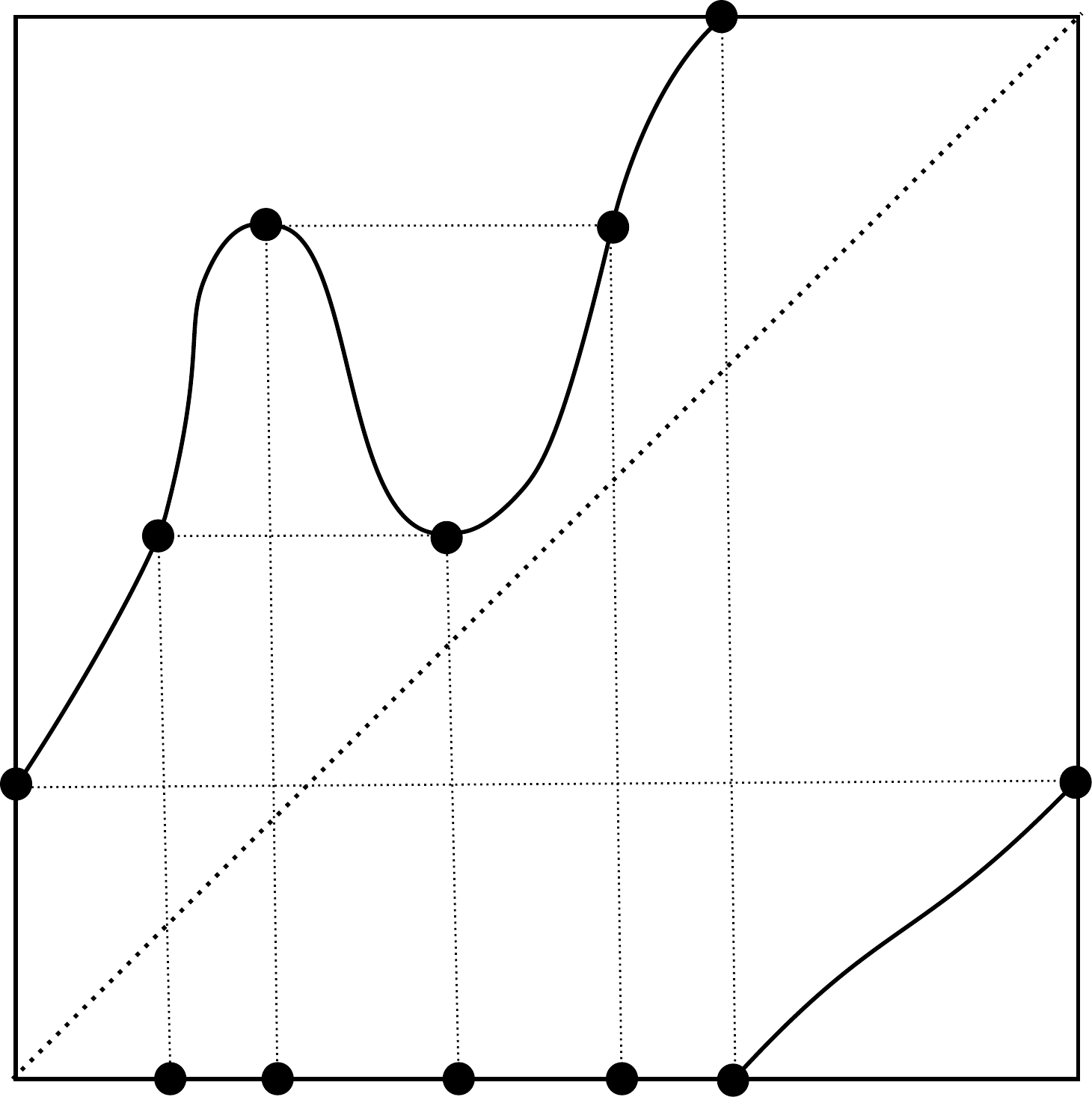}
\put(8, 60){$f$}
\put(33, 3){$\mathbb{I}$}
\put(13, -5,5){$d^-$}
\put(23, -5,5){$c^+$}
\put(40, -5,5){$c^-$}
\put(55, -5,5){$d^+$}
\end{overpic}
\medskip
\caption{The interval $\II$ is determined by the critical points $c^+$ and $c^-$. The points $d^+$ and $d^-$ correspond to the other preimage of each critical value.}
\label{fig1bimod}
\end{center}
\end{figure}

Many authors~\cite{boyland, chencinergt, misiurewicz} have shown that the upper rotation number $\rho^+$ of $\tilde f$ is equal to the rotation number of an endomorphism, the upper map $\tilde f_+$, defined by:
\begin{align*}
\tilde f_+(x)&=
\begin{cases}
\tilde f(\tilde c^+), \quad \text{if }
x\in (\tilde c^+,\tilde d^+),\\
\tilde f(x), \quad \text{if } x\in [\tilde d^+,\tilde c^++1],
\end{cases}\\
\tilde f_+(x+1)&=\tilde f_+(x)+1,
\quad \text{for any } x\in\mathbb{R}.
\end{align*}

\begin{figure}[h!]
\begin{center}
\begin{overpic}[scale=0.4]{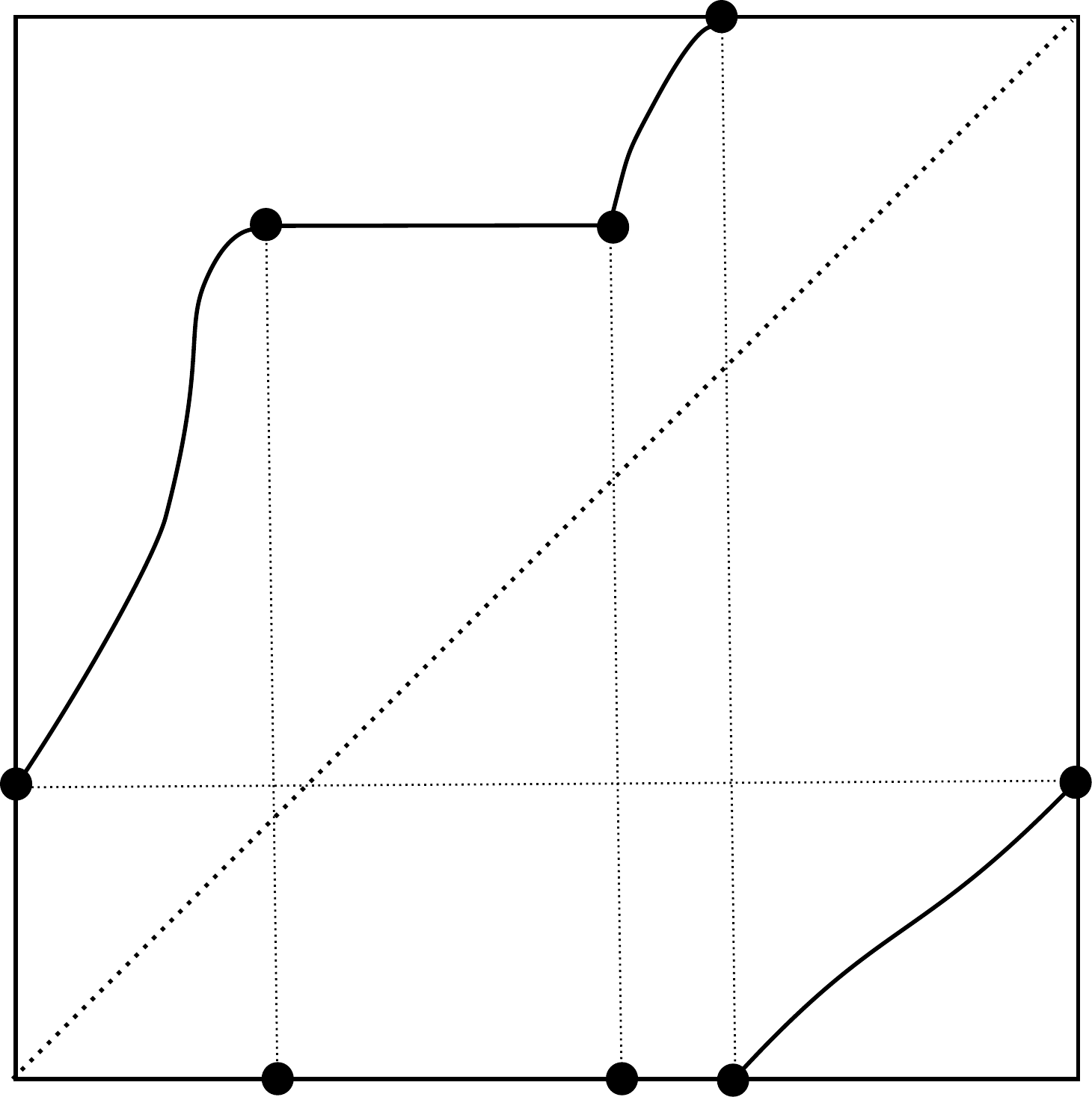}
\put(8, 60){$f_+$}
\put(23, -5,5){$c^+$}
\put(55, -5,5){$d^+$}
\end{overpic}
\medskip
\caption{The rotation number of the upper map $f_+$ equals $\rho^+$.}
\label{figupper}
\end{center}
\end{figure}

\subsubsection{}\label{bonordres} We get a continuous endomorphism $f_+$ on $\mathbb{T}^1$ (see Figure \ref{figupper}). It is constant on the interval $\II^+$. Since the rotation number $\rho^+$ is irrational, all the iterates $f_+^n(\II^+)$ for $n\in\mathbb{Z}$ are disjoint. The map $\tilde f_+$ is non-decreasing and the orbit of the interval $\tilde \II^+$ by $\tilde f_+$ is ordered as the orbits of the rotation with angle $\rho^+$:
\begin{equation}\label{bonordre}
\forall i,j,k\in\mathbb{Z},\quad\tilde f_+^i(\tilde \II^+)<\tilde f_+^j(\tilde \II^+)+k\Leftrightarrow(i-j)\rho^+<k.
\end{equation}

\subsubsection{} In the same way, one defines an increasing lower map $\tilde f_-$ whose rotation number is $\rho^-$. It is constant on the interval $\tilde \II^-$. 

\subsection{Continued fractions and combinatorics of rotations}\label{rotation} We recall some well known facts on rotations (see for example~\cite{hermanthesis}, chap. V).

\subsubsection{} The coefficients $(a_k^+)_{k\in\mathbb{N}}$ in the continued fraction representation of $\rho^+$ are defined by:
\begin{align*}
a_0^+&=[\rho^+],\quad &\rho_0^+=\rho^+-a_0^+,\\
a_k^+&=\left[\frac{1}{\rho_{k-1}^+}\right],\quad
&\rho^+_k=\frac{1}{\rho_{k-1}^+}-a_k^+.
\end{align*}

\subsubsection{} One associates to $\rho^+$ its approximations $\left(\frac{p_k^+}{q_k^+}\right)$. The numbers $(q_k^+)$ are defined by the following recurrence relations:
\begin{align*}
q_{-1}^+&=0,\quad q_0^+=1,\\
q_{k+1}^+&=a_{k+1}^+q_{k}^++q_{k-1}^+,\quad \text{for any } k\in\mathbb{N}.
\end{align*}
For any orbit of the rotation $R_{\rho^+}$, the integers $q_k^+$, $k\geq 1$ are the times when the orbit makes the closest return so far to the starting point. Since the map $f_+$ is semi-conjugated to the rotation $R_+=R_{\rho^+}$ by an increasing endomorphism of the circle, the same property holds for $f_+$.\\

\paragraph{\it Notation.} On the remainder of this paper, when there is no chance of confusion, we will omit the symbol $+$ for the sequences $(a_k^+)_{k\in\mathbb{N}}$ and $\left(\frac{p_k^+}{q_k^+}\right)$. We will simply denote them by $(a_k)_{k\in\mathbb{N}}$ and $\left(\frac{p_k}{q_k}\right)$ respectively.

\begin{proposition}[see~\cite{hermanthesis}, Section V.8]\label{bbo} For any $k\geq 0$, let $I$ be the interval $(f^{-q_k}_+(c^+),c^+)$ if $k$ is even and $(d^+,f^{-q_k}_+(c^+))$ if $k$ is odd. Then $n=q_{k+1}$ is the first time, larger that $-q_k$, the point $f^n_+(c^+)$ returns inside $I$.
\end{proposition}

\begin{corollary}\label{bound-q-n} For any $k\geq 0$ one has $|q_k\alpha-p_k|<q_{k+1}^{-1}$.
\end{corollary}

\begin{proof}[Proof of Corollary \ref{bound-q-n}] One considers the rotation $R_+$, whose orbits are ordered as for those of $f_+$.
The interval of length $|q_k\alpha-p_k|$ bounded by $0$ and $R^{q_k}(0)$ is disjoint from its $q_{k+1}-1$ first iterates
from the Proposition~\ref{bbo}. This concludes.
\end{proof}

\subsubsection{} Sometimes we will be mainly interested in the closest returns to the left (or to the right) of $c^+$: they are obtained for times
\begin{align}
&q_{2k-1}+lq_{2k},\; 0< l\leq a_{2k+1},\label{gauche}\\
\text{(or }
&q_{2k}+lq_{2k+1},\; 0\leq l< a_{2k+2}),
\text{ with } k\in\mathbb{N}.\label{droite}
\end{align}

The closest returns to the left (resp. to the right) for the backward iterates are obtained for times of the form~(\ref{droite}) (resp.~(\ref{gauche})). More precisely,

\begin{proposition}\label{rotationp} Let $t<t'$ be two successive times of the form given by~(\ref{droite}) (resp.~(\ref{gauche})). Then for any $0<n<t'$, the point $f_+^{-n}(c^+)$ is not contained in $(f^{-t}_+(c^+),c^+)$, (resp. $(d^+,f_+^{-t}(c^+))$).
\end{proposition}

\begin{proof} As above it is sufficient to prove the proposition for the rotation $R_+$. Let us fix some $k\in\mathbb{N}$ and let $T$ be the interval $[0,R_+^{-2q_{2k+1}}(0)]$. One knows that it contains $R_+^{-q_{2k+1}}(0)$. If $a_{2k+2}\geq 2$ one pulls back $T$ by $R_+^{q_{2k}+(l-1)q_{2k+1}}$ for $0<l<a_{2k+2}$ and gets that $$R_+^{-q_{2k}-lq_{2k+1}}(0)\in(R_+^{-q_{2k}-(l-1)q_{2k+1}}(0),R_+^{-q_{2k}-(l+1)q_{2k+1}}(0)).$$

This shows that in $[R_+^{-q_{2k}}(0),0]$,
\begin{equation*}
\begin{split}
R_+^{-q_{2k}}(0)&<
R_+^{-q_{2k}-q_{2k+1}}(0)<
R_+^{-q_{2k}-2q_{2k+1}}(0)<\cdots\\
\cdots &<R_+^{-q_{2k}-a_{2k+2}q_{2k+1}}(0)=R_+^{-q_{2k+2}}(0)<0.
\end{split}
\end{equation*}

(And this is true again in the case $a_{2k+2}=1$.) This is now enough to prove that for any $0\leq l<a_{2k+2}$ and $0<m<q_{2k+1}$,$$R_+^{-q_{2k}-lq_{2k+1}-m}(0)\not\in(R_+^{-q_{2k}-lq_{2k+1}}(0),0).$$

Proposition~\ref{bbo} shows for the rotation $R_+^{-1}$ that$$\forall 0<{\color{blue} m}<q_{2k+1},\quad R_+^m(0)\not\in (0,R_+^{q_{2k}}(0)).$$

Pulling back by $R_+^{q_{2k}+lq_{2k+1}}$ for $0\leq l<a_{2k+2}$, one gets$$\forall 0<m<q_{2k+1},\quad R_+^{-q_{2k}-lq_{2k+1}-m}(0)\not\in(R_+^{-q_{2k}-lq_{2k+1}}(0),R_+^{-lq_{2k+1}}(0)).$$

As $(R_+^{-q_{2k}-lq_{2k+1}}(0),0)\subset(R_+^{-q_{2k}-lq_{2k+1}}(0),R_+^{-lq_{2k+1}}(0))$, this concludes the proof.
\end{proof}

All the previous discussion can be obviously repeated for the lower map $\tilde f_-$ having rotation number $\rho_{-}$.

\subsection{Geometrical estimates}\label{S:geomest} The key estimates for our constructions have been proved by J.~Graczyk in~\cite{jacekmanuscript}, where the dynamics of upper maps is studied.

\begin{theorem}[J.~Graczyk]\label{thmgraczyk} Let $f_+$ be the upper map for some endomorphism $f$ which satisfies (A1) and (A2). 
Then there exists $C>0$ such that:
$$C^{-1}\prod_{0 { <} k\leq { n}} \left(1+\frac{a_{2k+1}}{\ell^+}\right)\leq|\log \dd (f_{+}^{q_{2n+1}}(c^+),c^+)|\leq C\prod_{0{ <} k\leq  { n}} 
\left(1+\frac{\ell^++1}{\ell^+-1}a_{2k+1}\right),$$ where $\ell^+$ is the order associated to $c^+$ that appears in (A2).
\end{theorem}

The estimates above, given for $f_{+}$, also hold for $f$ (since $f$ and $f_{+}$ coincide along the orbit of $c^{+}$).

\begin{remark}\label{rkgraczyk} In the same paper, J.~Graczyk gives the following estimate:
$$\log\left(\frac{|{ f_+^{-q_{2n}}(\II^+)}|}{\dd({ f_+^{-q_{2n}}(\II^+)},c^+)}\right)>C^{-1}\prod_{0{ <} k\leq n} \left(1+\frac{a_{2k+1}}{\ell^+}\right){ -C}.$$
\end{remark}
%\marginpar{\color{red} I have modified the formulae of the remark: Jacek's estimate involves the backward orbit of $\II^+$, not the forward.\\
%For me there is somethinstrange in this formula (and in Theorem~\ref{thmgraczyk} also: I would guess that the range for the products
%starts at $k=0$ rather the $k=1$ (as it is written in Jacek's paper. Indeed, otherwise $a_1$ does not appear in the formula.
%One should ask to Jacek?} 

Of course similar statements deal with lower maps.

\subsection{Extended upper map} 
%\subsubsection{} \label{debord} We denote by $\II=\pi([\tilde c^+,\tilde c^-])$ the maximal closed interval where $f$ is decreasing. We fix a closed interval $\hat\II\varsubsetneq\mathbb{T}^1$ whose interior contain $\II$. There is a partition of the form:$$\hat \II=\II_L\cup \II\cup \II_R,$$where the intervals $\II_L$ and $\II_R$ are adjacent to $\II$ and satisfy $\II_L<\II<\II_R$ in $\hat \II$.
%
%We also assume that$$\II_L\subset \II^-,\quad \text{and }\II_R\subset \II^+.$$

In addition to the upper map $f_+$, we will also use another upper map which is no more continuous, the extended upper map $g_+$ (or simply $g$, see Figure \ref{figextuppermap}).

\subsubsection{}\label{defm0} Let $M_0\geq 0$.
% which will be choosen large enough. See Sections \ref{defar}, \ref{defarp}, \ref{distorp2s} and \ref{regp}.
The following intervals have to be considered as a basis for the left neighborhoods of the point $c^+$:$$I_{k,l}^+=f_+^{-q_{2k}-lq_{2k+1}}(\II^+),\text{ for } k\geq M_0,\quad 0\leq l<a_{2k+2}.$$

\subsubsection{}\label{defpsi} Recall that the map $\tilde f$ is strictly increasing on $[\tilde f_+^{-q_{2M_0-2}}(\tilde c^+),\tilde c^+]$ and on $[\tilde c^-,\tilde d^+]$. As $\tilde f(\tilde c^+)=\tilde f(\tilde d^+)$, one defines an increasing homeomorphism $\tilde \Psi$ from $[\tilde f_+^{-q_{2M_0-2}}(\tilde c^+),\tilde c^+]$ onto some subinterval of $[\tilde c^-, \tilde d^+]$ by the condition:$$\forall x\in [\tilde f_+^{-q_{2M_0-2}}(\tilde c^+),\tilde c^+],\quad\tilde f(\tilde \Psi(x))=\tilde f(x).$$

By $\pi$, we get on $\mathbb{T}^1$ an homeomorphism $\Psi$ between some left neighborhoods of $c^+$ and $d^+$ respectively.

\subsubsection{}\label{defar} Let us now approximate the orbit of $c^+$ by a periodic orbit:

\begin{lemma}\label{def-a}
There exists a $q_{2M_0}$-periodic point $a^+\in f_+^{-q_{2M_0}+q_{2M_0-1}}(\II^+)$, having a unique iterate $b^+$ in $\II^+$.
\end{lemma}

\begin{remark} This gives $a^+=f_+^{-q_{2M_0}+q_{2M_0-1}}(b^+)=f^{q_{2M_0-1}}(b^+)$ and
$$\forall n\geq 1,\; f^n(a^+)\not \in (a^+,b^+).$$
For any $1\leq n\leq q_{2M_0}$, $f^n(b^+)\in f_+^{n-q_{2M_0}}(\II^+)$ so that the orbit of $b^+$ by $f$ in $\mathbb{T}^1$ is ordered as the orbits of the rotation with angle $\frac{p_{2M_0}}{q_{2M_0}}$.
\end{remark}

\begin{proof}[Proof of Lemma~\ref{def-a}] Note that $\Psi\circ f_+^{-q_{2M_0}}(\II^+)\subset \II^+$ and $f^{q_{2M_0}}(\Psi\circ f_+^{-q_{2M_0}}(\II^+))=\II^+$. Hence, $\Psi\circ f_+^{-q_{2M_0}}(\II^+)$ contains a fixed point $b^+$ for $f^{q_{2M_0}}$.
\end{proof}

\subsubsection{}\label{defarp} We will denote $$A_L^+=[a^+,c^+),\quad A^+_R=(b^+,d^+].$$
For $M_0$ large enough, one can assume$$A_R^+\subset \II^+, \quad A_L^+\cap \II^+=\emptyset\,\text{ and }\, A_L^+\subset \II^-.$$
One also sets$$A^+=\II^+\setminus A^+_R=[c^+,b^+],\quad \hat A^+=A^+_L\cup \II^+.$$
More generally we will consider the intervals$$A^+_R(k)=\Psi((f_+^{-q_{2k}}(d^+),c^+]),\quad \text{for $k\geq M_0$.}$$
The \emph{extended} upper map $g_+$ is defined by$$g_+(x)=\begin{cases}
f(x), \quad \text{ if }x\not \in A^+,\\
f(c^+),\quad \text{ if } x\in A^+.
\end{cases}$$

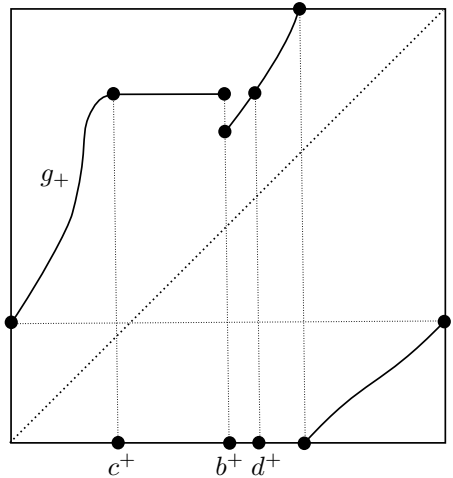
\begin{figure}[h!]
\begin{center}
\begin{overpic}[scale=0.4]{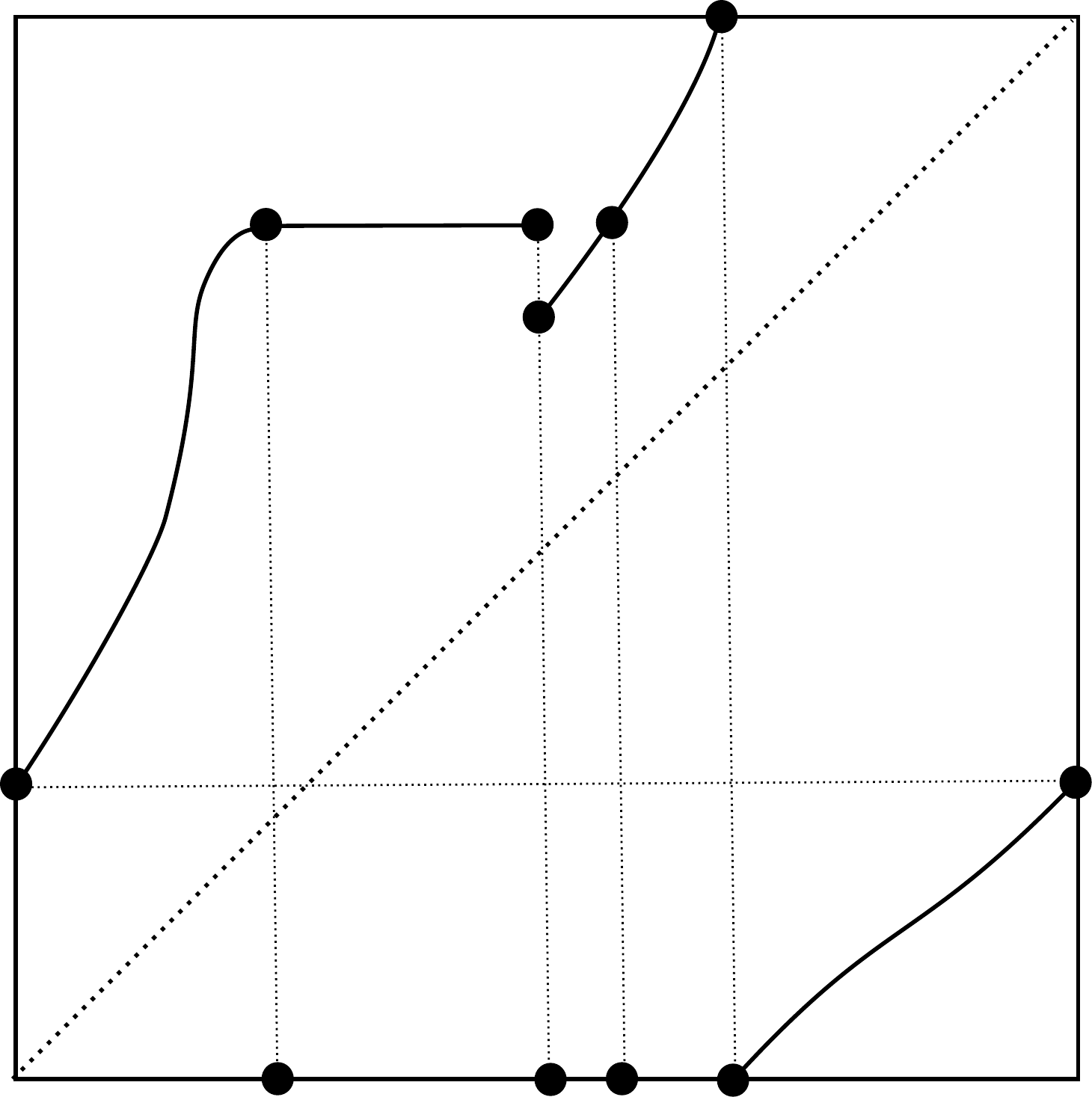}
\put(8, 60){$g_+$}
\put(23, -5,5){$c^+$}
\put(47, -5,5){$b^+$}
\put(55, -5,5){$d^+$}
\end{overpic}
\medskip
\caption{The extended upper map $g_+$ is discontinuous at the point $b^+$, which is a periodic point for $f$ with period $q_{2M_0}^+$ (see Lemma \ref{def-a}).}
\label{figextuppermap}
\end{center}
\end{figure}

\subsubsection{}\label{def-hat-II} In the same way, one defines intervals $A^-=[b^-,c^-]$, $\hat A^-=[d^-,a^-]$, $A^-_R=(c^-,a^-]$, $A^-_L=[d^-,b^-)$ and an extended lower map $g_-$. One also define $$A=A^+_L\cup \II\cup A^-_R.$$

If $M_0$ is large enough, this is a proper interval which contains $\II$.

Finally, one chooses a closed interval $\hat A=\hat \II$ whose interior contains $A$, and such that the connected components of $\hat A\setminus \II$ are contained in $\II^+$ and $\II^-$ respectively.

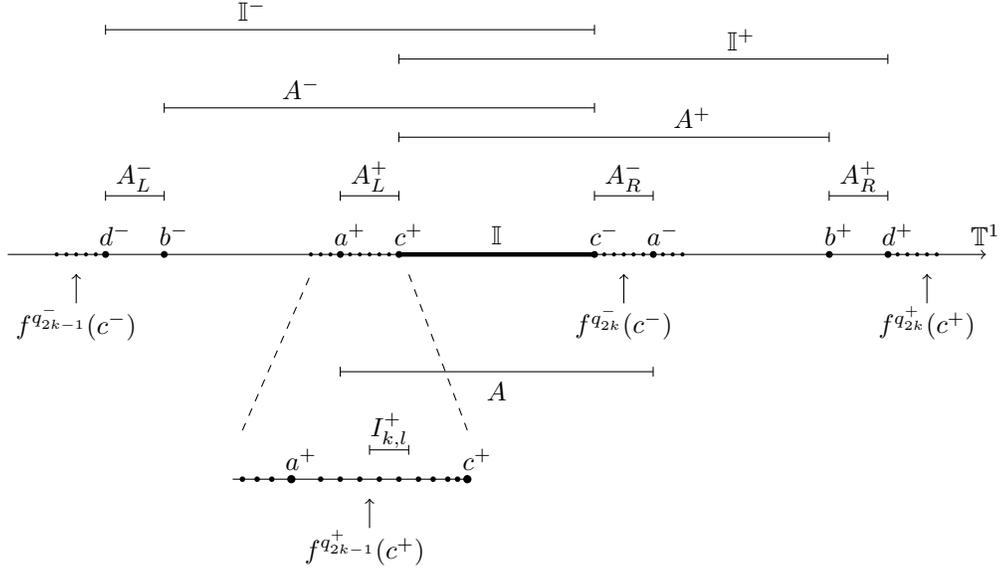
\begin{figure}[h!]
\centering
\begin{tikzpicture}[scale=1.3]
\coordinate (A) at (0,0);
\coordinate (B) at (10, 0);
\draw [->] (A) -- (B);
\node at (10, 0.2) {$\mathbb{T}^1$};

\coordinate (d-) at (1,0);
\coordinate (b-) at (1.6,0);
\coordinate (a+) at (3.4,0);
\coordinate (c+) at (4,0);
\coordinate (c-) at (6,0);
\coordinate (a-) at (6.6,0);
\coordinate (b+) at (8.4,0);
\coordinate (d+) at (9,0);

\draw [-, color=black, line width=0.6mm] (c+) -- (c-);
\node at (5, 0.2) {\color{black}$\mathbb{I}$};

\draw [-] (1, 0.6) -- (1.6,0.6);
\draw [-] (3.4, 0.6) -- (4,0.6);
\node at (1.3, 0.8) {$A_L^-$};
\node at (3.7, 0.8) {$A_L^+$};
\draw [-] (1,0.55 ) -- (1,0.65);
\draw [-] (1.6,0.55 ) -- (1.6,0.65);
\draw [-] (3.4,0.55 ) -- (3.4,0.65);
\draw [-] (4,0.55 ) -- (4,0.65);

\draw [-] (6, 0.6) -- (6.6,0.6);
\draw [-] (8.4, 0.6) -- (9,0.6);
\node at (6.3, 0.8) {$A_R^-$};
\node at (8.7, 0.8) {$A_R^+$};
\draw [-] (6,0.55 ) -- (6,0.65);
\draw [-] (6.6,0.55 ) -- (6.6,0.65);
\draw [-] (8.4,0.55 ) -- (8.4,0.65);
\draw [-] (9,0.55 ) -- (9,0.65);

\draw [-] (3.4, -1.2) -- (6.6,-1.2);
\node at (5, -1.4) {$A$};
\draw [-] (3.4,-1.15) -- (3.4,-1.25);
\draw [-] (6.6,-1.15 ) -- (6.6,-1.25);

\draw [-] (1.6, 1.5) -- (6,1.5);
\draw [-] (4, 1.2) -- (8.4,1.2);
\node at (3, 1.7) {$A^-$};
\node at (7, 1.4) {$A^+$};
\draw [-] (1.6,1.45 ) -- (1.6,1.55);
\draw [-] (6,1.45 ) -- (6, 1.55);
\draw [-] (4,1.15 ) -- (4,1.25);
\draw [-] (8.4,1.15) -- (8.4,1.25);

\draw [-] (1, 2.3) -- (6,2.3);
\draw [-] (4, 2) -- (9,2);
\node at (2.5, 2.5) {$\mathbb{I}^-$};
\node at (7.5, 2.2) {$\mathbb{I}^+$};
\draw [-] (1,2.25 ) -- (1,2.35);
\draw [-] (6,2.25 ) -- (6, 2.35);
\draw [-] (4,1.95 ) -- (4,2.05);
\draw [-] (9,1.95) -- (9,2.05);

\fill[black] (d-) circle (1pt) ;
\fill[black] (b-) circle (1pt) ;
\fill[black] (a+) circle (1pt) ;
\fill[black] (c+) circle (1pt) ;
\fill[black] (c-) circle (1pt) ;
\fill[black] (a-) circle (1pt) ;
\fill[black] (b+) circle (1pt) ;
\fill[black] (d+) circle (1pt) ;

\node at (1.1, 0.2) {$d^-$};
\node at (1.7, 0.2) {$b^-$};
\node at (3.5, 0.2) {$a^+$};
\node at (4.1, 0.2) {\color{black}$c^+$};
\node at (6.1, 0.2) {\color{black}$c^-$};
\node at (6.7, 0.2) {$a^-$};
\node at (8.5, 0.2) {$b^+$};
\node at (9.1, 0.2) {$d^+$};

\fill[black] (0.5, 0) circle (.6pt) ;
\fill[black] (0.6, 0) circle (.6pt) ;
\fill[black] (0.7, 0) circle (.6pt) ;
\fill[black] (0.8, 0) circle (.6pt) ;
\fill[black] (0.9, 0) circle (.6pt) ;

\fill[black] (3.1, 0) circle (.6pt) ;
\fill[black] (3.2, 0) circle (.6pt) ;
\fill[black] (3.3, 0) circle (.6pt) ;
\fill[black] (3.5, 0) circle (.6pt) ;
\fill[black] (3.6, 0) circle (.6pt) ;
\fill[black] (3.7, 0) circle (.6pt) ;
\fill[black] (3.8, 0) circle (.6pt) ;
\fill[black] (3.9, 0) circle (.6pt) ;

\fill[black] (6.1, 0) circle (.6pt) ;
\fill[black] (6.2, 0) circle (.6pt) ;
\fill[black] (6.3, 0) circle (.6pt) ;
\fill[black] (6.4, 0) circle (.6pt) ;
\fill[black] (6.5, 0) circle (.6pt) ;
\fill[black] (6.7, 0) circle (.6pt) ;
\fill[black] (6.8, 0) circle (.6pt) ;
\fill[black] (6.9, 0) circle (.6pt) ;

\fill[black] (9.1, 0) circle (.6pt) ;
\fill[black] (9.2, 0) circle (.6pt) ;
\fill[black] (9.3, 0) circle (.6pt) ;
\fill[black] (9.4, 0) circle (.6pt) ;
\fill[black] (9.5, 0) circle (.6pt) ;

\node at (0.7, -0.7) {$f^{q^-_{2k-1}}(c^-)$};
\draw [->] (0.7,-0.5 ) -- (0.7,-0.2);

%\node at (3.65, -0.7) {$f^{q^+_{2k-1}}(c^+)$};
%\draw [->] (3.7,-0.5 ) -- (3.7,-0.2);

\node at (6.3, -0.7) {$f^{q^-_{2k}}(c^-)$};
\draw [->] (6.3,-0.5 ) -- (6.3,-0.2);

\node at (9.4, -0.7) {$f^{q^+_{2k}}(c^+)$};
\draw [->] (9.4,-0.5 ) -- (9.4,-0.2);

\draw[-](2.3, -2.3)--(4.7, -2.3);
\fill[black] (2.9, -2.3) circle (1.2pt) ;
\fill[black] (4.7, -2.3) circle (1.2pt) ;
\node at (3, -2.1) {$a^+$};
\node at (4.8, -2.1) {$c^+$};
\draw [-, dashed] (2.4,-1.8 ) -- (3.1,-0.2);
\draw [-, dashed] (4.7,-1.8 ) -- (4.1, -0.2);

\fill[black] (2.4, -2.3) circle (.8pt) ;
\fill[black] (2.55, -2.3) circle (.8pt) ;
\fill[black] (2.7, -2.3) circle (.8pt) ;

\fill[black] (3.2, -2.3) circle (.8pt) ;
\fill[black] (3.4, -2.3) circle (.8pt) ;
\fill[black] (3.6, -2.3) circle (.8pt) ;
\fill[black] (3.8, -2.3) circle (.8pt) ;
\fill[black] (4, -2.3) circle (.8pt) ;
\fill[black] (4.2, -2.3) circle (.8pt) ;
\fill[black] (4.35, -2.3) circle (.8pt) ;
\fill[black] (4.5, -2.3) circle (.8pt) ;
\fill[black] (4.6, -2.3) circle (.8pt) ;

\draw[-](3.7, -2)--(4.1, -2);
\node at (3.9, -1.8) {$I_{k,l}^+$};
\draw [-] (3.7,-1.95 ) -- (3.7,-2.05);
\draw [-] (4.1,-1.95) -- (4.1, -2.05);

\node at (3.65, -3) {$f^{q^+_{2k-1}}(c^+)$};
\draw [->] (3.7,-2.8 ) -- (3.7,-2.5);

\end{tikzpicture}
\caption{Notation for Sections \ref{induces} to \ref{markovprop}.} 
\end{figure}

\section{Induced maps}\label{induces}

\subsection{ }\label{geneinduc} We begin with a general situation: let $I=[z_L,z_R]$ and $\hat I=[\hat z_L,\hat z_R]$ be proper intervals of $\mathbb{T}^1$ that contain $\II=\pi([\tilde c^+,\tilde c^-])$ such that $I\subset \interior(\hat I)$ and which satisfy for every integer $n\geq 1$:
\begin{align*}
&(i)f^n(c^+)\not\in [c^+,\hat z_R],\quad &(iii)f^n(z_L)\not\in (z_L,z_R],\\
&(ii)f^n(c^-)\not\in [\hat z_L,c^-],\quad &(iv)f^n(z_R)\not\in [z_L,z_R).
\end{align*}

For any point $x\in\mathbb{T}^1\setminus I$ one defines (when it exists) the smallest integer $N(x)\geq 1$ such that $f^{N(x)}(x)\in I$. In the other case, one sets $N(x)=\infty$.

\begin{proposition}\label{inducp} Let $I$ and $\hat I$ be as above. For any $x\in\mathbb{T}^1\setminus I$ such that $N=N(x)<\infty$, there exist some compact intervals $J\subset \hat J$ containing $x$ such that:
\begin{enumerate}
\item $\forall\,\,0\leq n<N, \; f^n(J)\cap I=\emptyset$;
\item the map $f^N$ is a homeomorphism from $J$ (resp $\hat J$) onto $I$ (resp. $\hat I$). Moreover, for any $y\in J$, $N(y)=N(x)$.
\end{enumerate}
\end{proposition}

Such an interval $J$ will be called a \textbf{return interval} with \textbf{extension} $\hat J$ and \textbf{order} $N$.

\begin{proof}[Proof of Proposition \ref{inducp}] Let $x\in\mathbb{T}^1\setminus I$ and let $[x_1,x_2]$ be the maximal compact interval containing $x$ where $f^N$ is monotone. As $f^m(x)\not\in \II$ for any $0\leq m< N$, the maps $f^n$ for $0\leq n\leq N$ are strictly increasing on $[x_1,x_2]$. Moreover, there exist some integers $1\leq n_1,n_2\leq N$ with $f^{N-n_i}(x_i)\in (\{c^-,c^+\})$. Thus, by assumption,$$f^{N-n_1}([x_1,x_2])\subset \mathbb{T}^1\setminus\interior(\II).$$

One deduces $f^{N-n_1}(x_1)=c^-$ and in a same way $f^{N-n_2}(x_2)=c^+$.

\begin{claim}
$f^N([x_1,x_2])\supset \hat I$.
\end{claim}
\begin{proof} When $\hat z_R\in [f^N(x),f^N(x_2)]$ and $\hat z_L\in [f^N(x_1),f^N(x)]$, the claim follows immediately.
Let us suppose by contradiction the first inclusion does not hold (the other case is similar).
By assumption, $f^N(x_2)\not\in[c^+,\hat z_R]$ so that$$f^N((x,x_2])\subset (z_L,c^+).$$

On the other hand $f^{N-n_2}(x)\not\in [z_L,c^+]$ and $f^{N-n_2}(x_2)=c^+$ so that$$f^{N-n_2}((x,x_2])\supset [z_L,c^+].$$

Hence what we get contradicts the assumptions:$$f^{n_2}(z_L)\in f^N((x,x_2])\subset (z_L,c^+).$$
\end{proof}

From the claim and $f^N(x)\in I$, one deduces that there are some compact intervals $J\subset \hat J$ that contain $x$ and are mapped onto $I$ and $\hat I$ respectively. If one assumes that $f^n(J)$ intersects $I$ for some $0\leq n <N$, since $f^n(x)\not\in I$ one would deduce either that $z_L\in (f^n(x),f^n(x_2)]$ or $z_R\in[f^n(x_1),f^n(x))$. Hence, $f^{N-n}(z_L)\in(z_L,z_R]$ or $f^{N-n}(z_R)\in [z_L,z_R)$. This is impossible so that$$\forall 0\leq n<N, \quad f^n(J)\cap I=\emptyset.$$
\end{proof}

Using the previous general setting and Proposition \ref{inducp} we give now the definition of first entry map to different intervals which will play a main role in the proof of our results.

\subsection{First return map to $\II$}\label{frmi} Section~\ref{geneinduc} applies with $I=\II$ and $\hat I=\hat \II$ (defined in Section~\ref{def-hat-II}). Because of the hypothesis that $\rho^+$ and $\rho^-$ are irrational, conditions $(i)$, $(ii)$, $(iii)$ and $(iv)$ of Subsection \ref{geneinduc} are satisfied. As a consequence, Proposition \ref{inducp} applies. Note that the result extends also for points $x\in I=\II$ with the same proof. (However for any $0<m\leq N(x)$, the map $f^m$ is strictly decreasing and $f^{N-n_1}(x_1)=c^+$, $f^{N-n_2}(x_2)=c^-$.) The integer $N(x)$ will be denoted by $N^0(x)$. The map $T^0:x\mapsto f^{N^0(x)}(x)$ defined on points $x\in\mathbb{T}^1$ such that $N^0(x)<\infty$ is called the \textbf{first entry map} or the \textbf{first return map} when it is restricted to $\mathbb{T}^1\setminus\II$ or to $\II$ respectively.

The set of points $x\in\mathbb{T}^1$ such that $N^0(x)$ is finite is a union of disjoint compact intervals with non-empty interior. Thus, one gets a family $\mathcal{N}$ of intervals of $\mathbb{T}^1$ and a map $N^0:\mathcal{N}\to\mathbb{N} $ defined as $N^0(I)=N^0(x)$ with $x\in I\in\mathcal{N} $. By Proposition \ref{inducp} the function $N_0$ is well defined. 

Observe that an interval $I\in\mathcal{N}$ is either contained in $\II$ or $\mathbb{T}^1\setminus \II$. The set of intervals $I\in\mathcal{N}$ contained in $\II$ will be denoted by $\mathcal{N}^0$.

\subsection{First entry map to $A^+$}\label{feap} One can consider the case $I=A^+$ and $\hat I=\hat A^+$, see Subsection \ref{defarp}. Being $\rho^+$ and $\rho^-$ irrational, conditions $(i)$, $(ii)$, $(iii)$ and $(iv)$ of Subsection \ref{geneinduc} are satisfied and Proposition \ref{inducp} applies.  The entry time is denoted by $N^+(x)$. The map $T^+:x\mapsto f^{N^+(x)}(x)$ defined on points $x\in \mathbb{T}^1\setminus A^+$ such that $N^+(x)<\infty$ is called the \textbf{first entry map} to $A^+$. It is also the first entry map to $A^+$ for the dynamics induced by $g_+$\,. As before we generate a family of intervals $\mathcal{M}^+$.

\begin{remark}\label{feapr} Consider in Proposition~\ref{inducp} 
the interval $J'$ such that $J\subset J'\subset \hat J$ and $f^{N^+(x)}(J')=\II^+=[c^+,\hat z_R]$. Then for any $0\leq n < N^+(x)$, the interval $f^n(J')$ does not intersect $A^+$. This is due to the fact that in this case, $z_L=c^+$.
\end{remark}

In the same way, one will consider on $\mathbb{T}^1\setminus A^-$ the first entry map $T^-$ to $A^-$ for $f$ or $g_-$\,. It is defined on a family of intervals $\mathcal{M}^-$ with return time $N^-:\mathcal{M}^-\rightarrow \mathbb{N}$.

\subsection{First entry map to $A$}\label{frma} The last induced map we will use is the first entry map to $A$ with $I=A$ and $\hat I= \hat A$, see Subsection \ref{def-hat-II} . As before, conditions $(i)$, $(ii)$, $(iii)$ and $(iv)$ of Subsection \ref{geneinduc} are satisfied and Proposition \ref{inducp} applies (Recall how the orbits of $c^-$ and $c^+$ are ordered on $\mathbb{T}^1$, see~(\ref{bonordre}) at Section~\ref{bonordres}.) We will denote by $N_A(x)$ the integer $N(x)$ and $T_A$ the first entry map to $A$. The set of points $x\in\mathbb{T}^1\setminus A$ where $N_A(x)<\infty$ decomposes as a union of disjoint compact intervals over a family $\mathcal{S}$. The set of point $x$ with $N_A(x)=\infty$ is contained in a maximal invariant set $K$ in $[a^+,a^-]$.

By a well-known result of Ma\~n\'e \cite{mane}, $K$ is hyperbolic. A classical result for $C^2$-maps (see~\cite{demelovanstrien}, Chapter III, Theorem 2.6) shows that $K$ has zero-Lebesgue measure. More precisely:

\begin{proposition}\label{frmap} There exist $C>0$ and $\kappa>1$ such that for any $n\in\mathbb{N}$,$$\lambda\{x\in\mathbb{T}^1\setminus A,\;N_A(x)>n\}<C. \kappa^{-n}.$$

In particular, $\mathcal{S}$ is a measurable partition and for Lebesgue-almost every $x\in\mathbb{T}^1\setminus A$, $N_A(x)$ is finite and the map $N_A:\mathcal{S}\rightarrow \mathbb{N}$ is summable.
\end{proposition}

In the following we will discuss general properties of the first return maps. Being the specific interval of definition irrelevant, we will call the first return map to any interval simply \textbf{induced map}.

\section{Distortion properties of the induced maps}\label{distprop} In order to control the distortion of the induced maps introduced in Section \ref{induces}, we state some classical results:

\subsection{Koebe principle}\label{distest} For any non-empty intervals $J$ and $\hat J$ which are strictly contained in $\mathbb{T}^1$ and such that $\closure(J)\subset \interior(\hat J)$, we define$$\Di (J,\hat J)=\frac{|J|}{\distance(J,\boundary(\hat J))},$$where $\distance(J,\boundary(\hat J))$ denotes the length of the smallest component of $\hat J\setminus \closure(J)$. The Koebe principle for interval maps proved in~\cite{graczykss2}, Proposition 1 (see also~\cite{graczykss1}) remains true for circle maps. We get for the endomorphism $f$ the following control on the distortion:

\begin{theorem}[Koebe principle, \cite{graczykss2}]\label{koebethm} There is a constant $\delta_0>0$ which satisfies the following property: for any non-empty intervals $J,\hat J\varsubsetneq \mathbb{T}^1$ such that $\closure(J)\subset \interior(\hat J)$ and for any $N\in\mathbb{N}$ such that $f^N$ in restriction to $\interior(\hat J)$ is a diffeomorphism, we have,$$\forall x,y\in J,\quad\frac{\D f^N(x)}{\D f^N(y)}\leq (1+\Di (f^N(J),f^N(\hat J)))^2\exp(\delta_0\sum_{n=0}^{N-1}|f^n(J)|).$$
\end{theorem}

\subsection{Hyperbolicity} We will need to show that some maps are hyperbolic. This was proved by R.~Ma\~n\'e in~\cite{mane} for one-dimensional $C^2$-maps. We state and prove here (Appendix \ref{appMa}) an analogous result for induced maps.

\begin{theorem}\label{maneinduce} Let $\mathcal{N}^0$ be a family of disjoint compact subintervals of $(0,1)$ with non-empty interior. Let $T:\mathcal{J}\rightarrow [0,1]$ be a map defined on $\mathcal{J}=\cup_{J\in\mathcal{N}^0}J$ that satisfies:

\begin{enumerate}
\item for any $J\in\mathcal{N}^0$, the restriction of $T$ on $J$ is a $C^1$-diffeomorphism onto $[0,1]$;
\item there exists a constant $D_d$ such that for any $J\in\mathcal{N}^0$, $$\forall x,y\in J,\; \frac{\D T(x)}{\D T(y)}\leq 1+D_d|T(x)-T(y)|;$$
\item any periodic orbit of $T$ is hyperbolic repulsive.
\end{enumerate}

Then, $T$ is hyperbolic: there exist some constants $C>0$ and $\kappa>1$ such that for any orbit $x, T(x), \cdots, T^{n-1}(x)$, in $\mathcal{J}$,$$|\D T^n(x)|\geq C.\kappa^n.$$
\end{theorem}

\subsection{Distortion of the induced maps}\subsubsection{} Let us consider $\varepsilon>0$ small and an induced map $T\in\{T^0, T^-,T^+,T_A\}$ for the intervals $I\subset \hat I$. For any return interval $J$ of $T$ with order $N$ and extension $\hat J$, we define $\Hat{\Hat{J}}\subset \hat J$ to be the unique compact interval contained in $\hat J$ such that both components of $T(\Hat{\Hat{J}}\setminus J)$ have length $\varepsilon|I|$.

\begin{proposition}\label{distorp1} If $\varepsilon>0$ is small enough, there exists some constant $D_1>0$ such that: for any induced map $T\in\{T^0, T^-,T^+,T_A\}$ and any return interval $J$ associated to $T$ with order $N$, we have,$$\forall x,y\in \Hat{\Hat{J}},\;\left|\frac{\D f^N(x)}{\D f^N(y)}\right|\leq 1+D_1\dd (f^N(x),f^N(y)).$$
\end{proposition}

\begin{proof} We note $D_m$ the maximum over (see Section~\ref{distest}):$$\Di (\Hat{\Hat{\II}},\hat \II), \;\Di (\Hat{\Hat{A}},\hat A^+),\; \Di (\Hat{\Hat{A}},\hat A^-), \;\Di (\Hat{\Hat{A}},\hat A).$$

Then, one defines (see Section~\ref{distest}) $K=(1+D_m)^{2}\exp(3\delta_0)$. By shrinking $\varepsilon$ again, one may assume
\begin{equation}\label{cc1}
2\varepsilon K\leq 1.
\end{equation}

We prove inductively that for any $0\leq n\leq N$,
\begin{enumerate}
\item\label{propi1} $|f^{N-n}(\Hat{\Hat{J}})|\leq 2|f^{N-n}(J)|$;
\item\label{propi2} $f^{n}$ has distortion bounded by $K$ on $f^{N-n}(\Hat{\Hat{J}})$:$$\forall x,y\in f^{N-n}(J_\varepsilon),\;\left|\frac{\D f^n(x)}{\D f^n(y)}\right|\leq K.$$
\end{enumerate}

Those properties are obvious for $n=0$ by definition of $\Hat{\Hat{J}}$. Let us assume that~\ref{propi2}) has been proved for any $n\leq n_0$. We first remark that~\ref{propi1}) is a direct consequence of~\ref{propi2}): we get from Koebe Theorem~\ref{koebethm} and~(\ref{cc1}),$$|f^{N-n_0}(\Hat{\Hat{J}}\setminus J)|\leq K\frac{|f^N(\Hat{\Hat{J}}\setminus J)|}{|f^N(J)|}{|f^{N-n_0}(J)|}\leq K2\varepsilon {|f^{N-n_0}(J)|}\leq {|f^{N-n_0}(J)|}.$$

Let us assume that $n_0\leq N-1$. We now prove~\ref{propi2}) for $n_0+1$. The sum $\sum_{n=1}^{n_0+1}|f^{N-n}(\Hat{\Hat{J}})|$ is bounded by $|f^{N-n_0-1}(\Hat{\Hat{J}})|+2\sum_{n=1}^{n_0}|f^{N-n}(J)|$ which is less that $3$. The quantity $\Di (f^N(\Hat{\Hat{J}}), f^N(\hat J))$ is bounded by $D_m$. Hence, Koebe Theorem~\ref{koebethm} gives the announced bound (recall how $K$ has been defined).

We take now some interval $[x,y]\subset \Hat{\Hat{J}}$.

By~\ref{propi2}), we get$$\sum_{k=0}^{N-1}\dd (f^k(x),f^k(y))\leq K\frac{\dd (f^N(x),f^N(y))}{|f^N(J)|}\sum_{k=0}^{n-1}|f^k(J)|\leq K\frac{\dd (f^N(x),f^N(y))}{|f^N(J)|}.$$

This gives the following estimate:
%\marginpar{\color{red} I don't understand why the estimate in the initial version was more complicate.}
\begin{equation*}
\begin{split}
\left|\frac{\D f^N(x)}{\D f^N(y)}\right|&\leq{(1+D_m)^2}\exp\left(\delta_0 K \frac{\dd (f^N(x),f^N(y))}{|f^N(J)|}\right)\\
&\leq 1+D_1\dd(f^N(x),f^N(y)),
\end{split}
\end{equation*}
for some new uniform constant $D_1>0$.
\end{proof}

\subsubsection{}\label{distorp2s} The distortion is also bounded when induced maps are composed:

\begin{proposition}\label{distorp2} There exists $D_2>0$ which satisfies: let $T_0,\cdots, T_n$ be a sequence of induced maps in $\{T^0,T^-,T^+,T_A\}$ and for any $0\leq k\leq n$ a return interval $J_k$ associated to $T_k$ with order $N_k$. One assumes furthermore that for any $0\leq k<n$,
\begin{equation*}
\begin{split}
J_{k+1}\subset f^{N_k}(J_k)\setminus\interior(\II),&\quad \hat J_{k+1}\subset f^{N_k}(\hat J_k),\\
T_k\in \{T^-,T^+\}\;\Longrightarrow\; &J_{k+1}\subset A_L^+\cup A_R^-.
\end{split}
\end{equation*}

Let us denote $J=J_0\cap T_0^{-1}(J_1)\cap\cdots\cap (T_{n-1}\circ\cdots\circ T_0)^{-1}(J_n)$. Then, the distortion of $T=T_n\circ\cdots\circ T_0$ on $J$ is bounded:$$\forall x,y\in J,\; \left|\frac{\D T(x)}{\D T(y)}\right|\leq D_2.$$
\end{proposition}

\begin{proof} For $M_0$ large enough,$$|A^-_R|,|A^+_L|\leq \frac{(1+D_1)^2}{4}|\II|^2.$$

By Proposition~\ref{distorp1}, for $k\geq 0$, $x\in J_k$
the derivative $DT_k(x)$ is bounded from below:
$$\D T_k(x)\geq (1+D_1)\frac{|T_k(J_k)|}{|J_k|}.$$
Since $J_{k+1}\cap \interior(\II)=\emptyset$,
the return map $T^0$ may only appear for $T_n$. By definition, $T_A$ can not appear for two consecutive times.
One deduces $DT_{k+1}\circ T_k\geq 4$ for any $0\leq k<n-1$. Consequently, $|J|$ decreases exponentially with $n$ and for $x,y\in J$:
\begin{equation*}
\begin{split}
\left|\frac{\D T(x)}{\D T(y)}\right|&\leq\prod_{k=0}^{n}\left|\frac{\D T_k(T_{k-1}\circ\cdots\circ T_0(x))}{\D T_k(T_{k-1}\circ \cdots \circ T_0(y))}\right|\\
&\leq\prod_{k=0}^{n} \left(1+D_1\dd (T_k\circ\cdots\circ T_0(x),T_k\circ\cdots\circ T_0(y))\right)\\
&\leq D_2:=(1+D_1)\prod_{k=0}^{n-1} (1+D_12^{-k}).
\end{split}
\end{equation*}
\end{proof}

\subsubsection{}\label{distortt} We will precise the previous proposition in the case all the maps we compose are $T^0$:

\begin{proposition}\label{distorp3} There exists a constant $D_3$ such that for any sequence of return intervals $J_0,\cdots, J_n$ in $\mathcal{N}^0$, the distortion of $(T^0)^n$ on $J=J_0\cap (T^0)^{-1}(J_1)\cap\cdots\cap (T^0)^{-n+1}(J_n)$ is bounded by $D_3$:$$\forall x,y\in J,\;\left|\frac{\D(T^0)^n(x)}{\D(T^0)^n(y)}\right|\leq 1+D_3\dd((T^0)^n(x),(T^0)^n(y)).$$
\end{proposition}

\begin{proof}[Proof of Proposition \ref{distorp3}] By Proposition~\ref{distorp1}, Theorem~\ref{maneinduce} proved in Appendix \ref{appMa} holds for $T^0$ on the interval $\interior \II$. Thus the length of $J$ decreases exponentially with $n$. One concludes as in Proposition~\ref{distorp1} and~\ref{distorp2}.
\end{proof}

\section{Markov properties of the first return map}\label{markovprop}

We will prove in this section that the maps constructed before are Markov maps:

\begin{definition}\label{markovdef}

\renewcommand{\theenumi}{\roman{enumi}}

A map $T:\interior(\II)\rightarrow\interior(\II)$ is a \textbf{Markov map} of $\interior(\II)$ if there exists a finite or countable family $\mathcal{N}=\{I_m\}$ of disjoint intervals in $\interior(\II)$ such that:
\begin{enumerate}
\item $\interior(\II)\setminus\cup_{I_m\in\mathcal{N}} I_m$ has zero Lebesgue measure.
\item\label{induce1} For any $I_m\in\mathcal{N}$, the map $T$ is a $C^1$-diffeomorphism from $\interior(I_m)$ onto $\interior(\II)$;
\item\label{markovdef2} The distortion is bounded: for some $D_0>0$, any $n\in\mathbb{N}$, and any interval $J\subset\interior(\II)$ such that for all $1\leq j\leq n$, $T^j(J)$ is contained in some interval of $\mathcal{N}$, one has$$\forall x,y\in J,\quad\frac{\D T^n(x)}{\D T^n(y)}\leq 1+D_0|T^n(y)-T^n(x)|.$$
\end{enumerate}
\end{definition}

One can find more general definitions of Markov maps. However, Definition \ref{markovdef} is enough for our purposes in this paper.

\bigskip

The aim of this section is to prove two propositions which will play a key role in the proof of our main theorems. Using the notation introduced in \ref{frmi}, we claim:

\begin{proposition}\label{fonda} The map $T^0$ is a Markov map of $\interior(\II)$ associated to the measurable partition $\mathcal{N}^0$.
\end{proposition}

Let us remark that Proposition \ref{fonda} holds for any map in $\Bimod$, without need of conditions \ref{C-} and \ref{C+}. However:

\begin{proposition}\label{fondb} The map $N^0$ associated to the family $\mathcal{N}^0$ is summable if and only if both conditions \ref{C-} and \ref{C+} are satisfied.
\end{proposition}

We recall that the definition of \emph{summable} was given in Section \ref{preliminaries} (Definition \ref{premdef}).

The proof of Proposition \ref{fonda} and Proposition \ref{fondb} will be divided into different subsections. All of them will contain in the beginning a short overview on its content.

\subsection{Decompositions}\label{decompositions} The goal of this subsection is to prove that the partition $\mathcal{M}^+$ (see Section \ref{feap}, before Remark \ref{feapr}, for its definition) gives a measurable partition of $A_L^+$ (see Proposition~\ref{propreg} below). For this aim we construct a primary decomposition of $A_L^+$ and successive refinements of it until to get a regular decomposition. Let us then start giving some basic definitions. 
\subsubsection{Basic definitions}
We define the regular intervals for the map $g=g_+$.

\begin{definition}\label{reg1} Let $J$ be an interval contained in $A_L^+$ with non-empty interior and $N\geq 1$ an integer such that
\begin{equation}\label{return}
\forall 0\leq n\leq N-1, \quad g_+^n(J)\cap A^+=\emptyset.
\end{equation}

\begin{itemize}
\item If $J$ is open, if the interval $(x_1,x_2)=g_+^N(J)$ is contained in $A_L^+$ and if it satisfies $x_2=c^+$ and $x_1=f^m_+(c^+)$ or $x_1=f^m_+(a^+)$ for some number $m\in\mathbb{N}$, then the interval $J$ is a \textbf{gap}.
\item If $J$ is closed and $g_+^N(J)=\II^+$, the interval $J$ is a \textbf{rough interval}.
\item If $J$ is closed and $g_+^N(J)=A^+$, the interval $J$ is a \textbf{regular interval} (for $g_+$).
\end{itemize}

From this definition, the integer $N$ is uniquely defined and will be called the \textbf{order} of $J$ and denoted by $N(J)$.
\end{definition}

In any case on $J$ the map $g_+^{N(J)}$ is an homeomorphism onto its image. Note also that if $J$ is a regular interval, its iterates $J, \cdots,f^N(J)$ are disjoint. Distortion on regular intervals and gaps can be controlled thanks to Proposition~\ref{distorp1}.

\subsubsection{} The following proposition is a direct consequence of Section~\ref{feap} and Remark~\ref{feapr}.

\begin{proposition}\label{propextend}
\begin{enumerate}
\item Let $J\subset A_L^+$ be a regular interval. Then there exists a unique compact interval $\hat J$ containing $J$ that is sent by $f^{N(J)}$ homeomorphically onto $\hat A^+$.
\item\label{propextend2} The map that associates to any rough interval $J$ the regular interval $J'\subset J$ defined by $g_+^{N(J)}(J')=A^+$ is a bijection between the set of rough interval and the set of regular intervals.
\end{enumerate}
\end{proposition}

\subsubsection{} Gaps are very small with respect to regular intervals.

\begin{lemma}\label{ptes} There exists $\eta\in (0,1)$ such that for any gap $T\subset A^+_L$ and any regular interval $J$ that are adjacent with $T<J$, then,
$|T|<\eta |J|$.

The constant $\eta$ can be chosen arbitrarily small if $M_0$ is large enough, see subsection \ref{defm0}.
\end{lemma}

\begin{proof} To prove this lemma, we first remark that on $T\cup J$ the distortion of $f^{N(J)}$ is bounded by $D_1$ (Proposition~\ref{distorp1}). Consequently,$$\frac{|T|}{|J|}<(1+D_1)\frac{|A_L^+|}{|A^+|}=\eta.$$

One chooses $\eta$ small by taking $|A_L^+|$ small.
\end{proof}

\subsubsection{} We are now able to give the definition of {decompositions}:

\begin{definition}\label{dec}

\renewcommand{\theenumi}{\roman{enumi}}

A \textbf{decomposition} of $A_L^+$ is a partition $\mathcal{P}$ of $\interior(A_L^+)$ in intervals such that:
\begin{enumerate}
\item any interval $J\in\mathcal{P}$ is either a gap, a rough or a regular interval of $A_L^+$;
\item outside any neighbourhood of $c^+$, the partition $\mathcal{P}$ is finite.
\item \label{dec3} for any $J,J'\in\mathcal{P}$, $J<J'$ (for the order on $A_L^+$) implies $N(J)\leq N(J')$.
\end{enumerate}

A decomposition is \textbf{regular} if it does not contain any rough interval. A decomposition is \textbf{summable} if the sum $\sum_{J\in \mathcal{P}}N(J)|J|$ is finite.
\end{definition}

\begin{remarks}\label{decr}

\begin{enumerate}
\item From the definition we get that for any partition $\mathcal{P}$ there is an open interval in $\mathcal{P}$ whose boundary contains $a^+$, the left endpoint of $A_L^+$. Any interval $J$ in $\mathcal{P}$ is adjacent to an other interval $J'\in\mathcal{P}$ with $J<J'$. If $\boundary(J)$ does not contain $a^+$, $J$ is also adjacent to an interval $J''\in\mathcal{P}$ with $J''<J$. Moreover $J'$ (resp. $J''$) is open if and only if $J$ is closed.
\item\label{decr2} If $J$ is a gap then $N(J)=N(J')$.
 %Thus Property~\ref{dec3}) has only to be checked on closed intervals of $\mathcal{P}$.
\end{enumerate}
\end{remarks}

We want to prove that $\mathcal{M}^+$ gives a measurable partition of $A_L^+$, that is, we want to prove that $\lambda\big(A_L^+\setminus\cup_{J\in\mathcal{M}^+}J\big)=0$ (see Proposition \ref{propreg} below). With this purpose, we proceed in $3$ steps:

\begin{description}
\item[Step 1] We define a first decomposition $\mathcal{P}_0$ of $A_L^+$, called the primary decomposition.
\item[Step 2] For each decomposition $\mathcal{P}$, we construct a refined decomposition $\mathcal{P'}$.
\item[Step 3] By successive refinements of $\mathcal{P}_0$ we get a regular decomposition $\mathcal{P^+}$, and prove that $\mathcal{M}^+$ gives a measurable partition of $A_L^+$.
\end{description}

The same arguments apply for the partition $\mathcal{M}^-$.

\subsubsection{Step 1. The primary decomposition}\label{pint} We will build a partition of $\interior(A_L^+)$ that does not contain any regular interval. We have introduced in~\ref{defm0} the compact intervals $I_{k,l}^+$ for $k\geq M_0$ and $0\leq l< a_{2k+2}$. They are contained in $A^+_L$ and from the description of Section~\ref{rotation} about dynamics of rotations, we get by definition that they are rough intervals.

\begin{lemma} The connected components of $A_L^+\setminus\bigcup I_{k,l}^+$ are gaps of $A_L^+$.
\end{lemma}

\begin{proof} Let $T$ be such a connected component and $I_{k,l}^+$ the adjacent interval to its right ($T<I_{k,l}^+$). We note $n=q_{2k}+lq_{2k+1}$ the order of $I_{k,l}^+$. Let us first assume that $\boundary(T)$ does not contain $a^+$. By Proposition~\ref{rotationp}, $T$ does not meet any interval of the form $f_+^{-s}(\II^+)$, with $0\leq s\leq n$. Hence for any $0\leq s\leq n$, $f^s(T)$ does not meet $A^+$ and $f^n(T)$ is adjacent to $A^+$. Note that $f^n(T)$ is equal to $[f^{q_{2k+1}}_+(c^+),c^+)$ (if $l>0$) or to $[f_+^{q_{2k-1}}(c^+),c^+)$ (if $l=0$). %In the first case $k\geq M_0-1$ and in the second $k\geq M_0$.
As $a^+<f_+^{q_{2M_0-1}}(c^+)$ in $A_L^+$, in both situations $f^n(T)\subset A_L^+$ and $T$ is a gap. If $a^+$ belongs to $\boundary(T)$, then $n=q_{2M_0}$ and $T$ intersects $f_+^{-q_{2M_0}+q_{2M_0-1}}(\II^+)$. The only time $0\leq s\leq n$ that $f^s(T)$ intersects $\II^+$ occurs at $s=q_{2M_0}-q_{2M_0-1}$. Note that the left endpoint of $f^{s}(T)$ is $b^+$ so that $f^{s}(T)$ does not intersects $A^+$. Moreover, since $a^+$ is $q_{2M_0}$-periodic, $f^n(T)=\interior(A^+_L)$ and $T$ is a gap.
\end{proof}

The following proposition defines the \textbf{primary decomposition} $\mathcal{P}_0$.

\begin{proposition} The partition $\mathcal{P}_0$ of $A_L^+$ in intervals $I_{k,l}^+$ and connected components of $A_L^+\setminus \bigcup I_{k,l}^+$ is a decomposition.
\end{proposition}

\begin{proof} It remains to show that the order is monotone: by construction it is obvious that in $A_L^+$, if $I_{k,l}^+<I_{k',l'}^+$ then $N(I_{k,l}^+)\leq N(I_{k',l'}^+)$. 
By Remark~\ref{decr}.\ref{decr2}, property~\ref{dec3}) of Definition~\ref{dec} is satisfied.
\end{proof}

Recall that we have introduced an homeomorphism $\Psi$ at Section~\ref{defpsi} between some left neighbourhoods of $c^+$ and $d^+$. We define the \textbf{primary intervals} of $\II^+$ as image $I'_{k,l}$ by $\Psi$ of intervals $I_{k,l}^+$ with order larger or equal to $ q_{2M_0}$ or as connected components of $A_R^+\setminus \cup I'_{k,l}$. Note that $A^+$ and the primary intervals of $\II^+$ define a partition of $\II^+\setminus\{d^+\}$.

\subsubsection{Step 2. Decomposition's refinement.} Let $\mathcal{P}$ be a decomposition. Our aim here is to give a construction that associates to $\mathcal{P}$ a ``finer decomposition", $\mathcal{P}'$. Some intervals of $\mathcal{P}$ remain unchanged:
\begin{itemize}
\item the (unique) gap $T\in\mathcal{P}$ whose boundary contains $a^+$ remains in $\mathcal{P}'$;
\item let $J$ and $T$ be two adjacent intervals of $\mathcal{P}$ with $J<T$ and assume that $J$ is a regular interval and $T$ a gap. Then $J$ and $T$ remain in $\mathcal{P}'$.
\end{itemize}
In the following we consider two adjacent intervals $J$ and $T$ of $\mathcal{P}$ with $J<T$ and assume that $J$ is a rough interval and $T$ a gap. We will explain how to decompose $J\cup T$ in intervals that will belong to $\mathcal{P'}$.

\subsubsection{}\label{co1} We introduce some rough interval $J'$ adjacent to $T$ with $T<J'$: let $\bar J$ be the interval in $\mathcal{P}$ that is adjacent to $T$ with $T<\bar J$. Either $\bar J$ is a rough interval and $J'=\bar J$ or $\bar J$ is a regular interval and $J'$ is the rough interval associated to $\bar J$. Recall that by definition of decompositions, $N(J')=N(T)>N(J)$.

\begin{lemma} The interval $g_+^{N(J)}(J')$ has the following form:$$g_+^{N(J)}(J')=f_+^{N(J)-N(J')}(\II^+).$$

Moreover either $N(J')-N(J)=q_{2M_0-1}$ or there exists some integers $k\geq M_0$ and $0< l\leq a_{2k+1}$ satisfying $N(J')-N(J)=q_{2k-1}+lq_{2k}$.
\end{lemma}

\begin{proof} Let us consider the interval $I=T\cup J'$. By assumption the intervals $f^n(I)$ for $0\leq n< N(J')$ do not meet $A^+$. Thus the restriction of the maps $f^n$ for $0\leq n\leq N(J')$ on $I$ are increasing homeomorphisms.

The left endpoint of $f^{N(J)}(I)$ (which also belongs to $f^{N(J)}(J')$) is $d^+$. Its iterates $f^{n}(d^+)$ for $n\geq 1$ never meet $\II^+$. This shows that the intervals $f^{n}(I)$ for $N(J)\leq n< N(J')$ do not meet $\II^+$. As $f^{N(J')}(J')=\II^+$, one gets in particular the first part of the proposition.

It implies also that $f_+^{-n}(\II^+)$ does not meet $f^{N(J)}(T)$ for $0\leq n< N(J')-N(J)$.
Since $f^{N(J)}(T)$ is adjacent to $f_+^{N(J)-N(J')}(\II_+)$,
by Proposition~\ref{rotationp}, there exist some integers $k\geq 0$ and $0< l\leq a_{2k+1}$ such that$$N(J')-N(J)=q_{2k-1}+lq_{2k}.$$

As $T$ is a gap, we get $f^{N(J')}(T)\subset A_L^+$ and in $A_L^+$,$$a^+\leq f^{N(J')-N(J)}(d^+)=f^{q_{2k-1}-lq_{2k}}(c^+).$$

Note also that since $a^+\in f_+^{-q_{2M_0}+q_{2M_0-1}}(\II^+)$ (see Section~\ref{defar}),
$$f^{q_{2M_0-1}-q_{2M_0-2}}(c^+)<a^+<f^{q_{2M_0-1}}(c^+).$$

Consequently the smallest possible value for $N(J')-N(J)$ is $q_{2M_0-1}$. In the other cases $k\geq M_0$.
\end{proof}

\subsubsection{}\label{co2} The map $g_+^{N(J)}$ induces an homeomorphism $h$ from $J\cup T$ onto $\II^+\cup g_+^{N(J)}(T)$. In the case $N(J')-N(J)= q_{2M_0-1}$, we take the following partition of $\II^+\cup g_+^{N(J)}(T)$:$$\{A^+, (\II^+\cup g_+^{N(J)}(T))\setminus A^+\}.$$

Otherwise, $N(J')-N(J)=q_{2k-1}+lq_{2k}$, $k\geq 2M_0$ $0<l\leq a_{2k+1}$. Then, we consider the partition with the following intervals:
\begin{itemize}
\item the interval $A^+$,
\item the primary intervals of $\II^+$ with order less or equal to $q_{2k}$ (recall Section~\ref{pint} for the definition of primary intervals of $\II^+$),
\item the interval $A_R^+(k)\cup  g_+^{N(J)}(T)$ with $A_R^+(k)=\Psi((f_+^{-q_{2k}}(d^+),c^+])$ from~\ref{defarp}.
\end{itemize}
By pulling back by $h$ this partition on $J\cup T$, we define the new intervals of the decomposition $\mathcal{P}'$:

\begin{itemize}
\item the interval $h^{-1}(A^+)$ is a regular interval;
\item the pulling back by $h$ of the primary intervals of $\II^+$ are gaps or rough intervals with order strictly between $N(J)$ and $N(J')$;
\item the interval $T$ has been extended to $T'=h^{-1}(A^+_R(k))\cup T$ (or to $T'=h^{-1}(A_R^+)\cup T$ if $N(J')-N(J)= q_{2M_0+1}$).
\end{itemize}

\subsubsection{} We have defined the new partition $\mathcal{P}'$. Now we prove:

\begin{proposition} The partition $\mathcal{P}'$ is a decomposition of $A^+_L$.
\end{proposition}

\begin{proof} We consider the situation described in Sections~\ref{co1} to~\ref{co2} and show that $T'$ is a gap (with order $N(J')$). The other parts of Definition~\ref{dec} will then be easily satisfied. In the particular case $N(J')-N(J)= q_{2M_0-1}$, one gets immediately that the intervals $g_+^n(A_R^+)$ for $1\leq n\leq q_{2M_0+1}$ do not meet $\II^+$
(see Section~\ref{defar}). Since $a^+$ is $q_{2M_0}$-periodic, we have $f^{q_{2M_0-1}}(T')=\interior(A_L^+)$.

We now suppose $N(J')-N(J)=q_{2k-1}+lq_{2k}>q_{2M_0-1}$.
By construction for $0\leq n\leq N(J)$, the interval $g_+^n(T')$ does not meet $A^+$. As $g_+^{N(J)}(T')=A_R^+(k)\cup T$, and as $T$ is a gap we only consider the intervals $g_+^n(A_R^+(k))$ for $0\leq n\leq N(J')-N(J)$.
Since $N(J')-N(J)<q_{2k}+q_{2k+1}$, they do not meet $\II^+$, by Proposition~\ref{rotationp}.

Note that the left endpoint of $g^{N(J')-N(J)}(A_R(k))$ is $f_+^{q_{2k-1}+(l-1)q_{2k}}(c^+)$, and belongs to $A_L^+$. Thus $T'$ is a gap.
\end{proof}

\subsubsection{Step 3. The regular decomposition.}

\begin{proposition} There exists a regular decomposition $\mathcal{P}^+$ of $A_L^+$.
\end{proposition}

\begin{proof} Let us call $\mathcal{P}_0$ the primary decomposition of $A_L^+$. By refining inductively the decomposition, one gets a sequence of decompositions $(\mathcal{P}_k)_{k\in\mathbb{N}}$. One sees easily that for $k$ large the intervals with bounded order are the same in all decompositions $\mathcal{P}_k$. Two successive decompositions $\mathcal{P}_k$ and $\mathcal{P}_{k+1}$ can not contain a same gap or rough interval. Hence, the set of intervals $J$ of $A_L^+$ such that there exist some $m$ with $J\subset \mathcal{P}_k$ for any $k\geq m$ is a regular decomposition of $A_L^+$.
\end{proof}

\subsubsection{}\label{regp} Note that regular intervals belong to $\mathcal{M}^+$ (see Section~\ref{feap}). Hence they are disjoint. One shows now that $\mathcal{M}^+$ gives a measurable partition of $A_L^+$.

\begin{proposition}\label{propreg} The set of regular intervals for $g_+$ in $A_L^+$ defines a measurable partition.
\end{proposition}
%
%The proof uses the following lemma.
%
%\begin{lemma}\label{ptes0} There exists $0<\eta<1$ such that for any regular decomposition $\mathcal{M}$ of $A^+_L$ and for any interval $(x_1,c^+)\subset A^+_L$ with $x_1=f^m_+(c^+)$, $m\in\mathbb{Z}$,
%\begin{equation}\label{reglarg}
%\sum_{\substack{J \text{ gap of }\mathcal{M},\\
%J\subset (x_1,c^+)}}|J|
%\leq \eta |c^+-x_1|.
%\end{equation}
%
%The constant $\eta$ can be chosen arbitrarily small if $M_0$ is large enough.
%\end{lemma}
%
%The interval $(x_1,c^+)$ is partitioned by $\mathcal{M}$ in regular intervals and gaps. Estimate~(\ref{reglarg}) comes from Lemma~\ref{ptes}.

\begin{proof}[Proof of Proposition~\ref{propreg}] We define inductively a sequence of measurable partitions $\mathcal{M}_k$ of $A_L^+$ whose elements are regular intervals or gaps. First $\mathcal{M}_0$ is a regular decomposition of $A_L^+$. If $J\in\mathcal{M}_k$ is regular, one sets $J\in\mathcal{M}_{k+1}$. If $J\in\mathcal{M}_k$ is a gap, one refines it by pulling back by $f^{N(J)}$ the partition $\mathcal{M}_0$: for any interval $J'\in\mathcal{M}_0$ such that $f^{N(J)}(J)\cap J'\not =\emptyset$, one sets $J''=J\cap g_+^{-N(J)}(J')$ and $J''\in\mathcal{M}_{k+1}$. Either $J'$ is a gap (and one checks easily that $J''$ is a gap) or $J'$ is a regular interval. In this second case, we consider $(x_1,c^+)=f^{N(J)}(J)$. By Definition~\ref{reg1}, $x_1=f_+^m(c^+)$ or $x_1=f_+^m(a^+)$ for some $m\in\mathbb{Z}$. Hence, for any $n\in\mathbb{N}$, $f^n(x_1)\not \in \interior(A^+)$. As $J'$ returns to $A^+$ homeomorphically, one gets $x_1\not\in \interior(J')$ and $J'\subset f^{N(J)}(J)$. Consequently, $J''$ is regular. Since $f^{N(J)}$ is a diffeomorphism for any gap $J\in\mathcal{M}_k$, and since $\mathcal{M}_0$ is a measurable partition, the collection $\mathcal{M}_{k+1}$ is a measurable partition of $A_L^+$.

Any $J\in\mathcal{M}_k$ decomposes into regular intervals and gaps of $\mathcal{M}_{k+1}$. Lemma~\ref{ptes} gives$$\sum_{\substack{J'' \text{ gap of } \mathcal{M}_{k+1},\\J''\subset J}} |J''|\leq \eta (1+D_1) |J|.$$

One can assume that $\eta$ is small and $\eta (1+D_1)<1$. This implies that
\begin{equation}\label{estgap}
\sum_{J \text{ gap of }
\mathcal{M}_{k}}|J|\leq (\eta(1+D_1))^n\sum_{J \text{ gap of }
\mathcal{M}_{0}}|J|
\underset{n\rightarrow \infty}\longrightarrow0.
\end{equation}

This ends the proof.
\end{proof}

\subsubsection{}\label{regm} The same constructions could be done for the lower map $g_-$. There exists a regular decomposition $\mathcal{P}^-$ of $A_R^-$. One defines also regular intervals of $A_R^-$ for $g_-$. They belong to $\mathcal{M^-}$ and defines a measurable partition of $A_R^-$.

\subsection{The summability conditions}\label{sumcond} 

We assume in this subsection that Condition \ref{C+} is satisfied and we prove into three steps that the measurable partition $\mathcal{M}^+$ of $A_L^+$ is summable:

\begin{description}
\item[Step 1] The primary decomposition as defined at~\ref{pint} is summable.
\item[Step 2] The summability is preserved by refinement.
\item[Step 3] The measurable partition $\mathcal{M}^+$ of $A_L^+$ is summable. 
\end{description}

The same arguments apply for the lower map $g_-$ and the partition $\mathcal{M}^-$.

\subsubsection{Step 1. Summability of the primary decomposition}\label{sc1}
Let us start with a Lemma 
\begin{lemma}\label{in1}
There exists a constant $\gamma>1$ such that for all small $\epsilon$ and for $k$ large enough
\begin{equation*}
\begin{split}
\sum_{l=1}^{a_{2k+2}-1}(q_{2k}+lq_{2k+1})|I_{k,l}^+|\leq
&\left(\sum_{n\geq 0}(n+2)\gamma^{-n}\right)q_{2k+1}|I_{k,1}^+|+\\
&\qquad \left(1+\frac{1-\varepsilon}{1+\varepsilon}\gamma\right).q_{2k+2}\dd(f^{q_{2k+1}}(c^+),c^+).
\end{split}
\end{equation*}
\end{lemma}
\begin{proof}
 We consider the map $F=f^{q_{2k+1}}=f^{q_{2k+1}-1}\circ f$ on $U=(f_+^{-q_{2k}-q_{2k+1}}(d^+),c^+)$. On one hand, we remind that near $\tilde c^+$ we have $\tilde f=\tilde f(\tilde c^+)-|\psi^+|^{\ell^+}$ (from assumption (A1)). On the other hand, from Koebe Theorem~\ref{koebethm}, the distortion of $f^{q_{2k+1}-1}$ on the interval $f(U)$ is arbitrarily small if $k$ is large enough: by Proposition~\ref{rotation}, the iterates $f^k(U)$ for $0\leq k \leq q_{2k+1}-1$ are disjoint; there exists also an interval $\hat I\supset f(U)$ that is sent homeomorphically by $f^{q_{2k+1}-1}$ onto $\hat A^+$ ($\hat I$ is the extension of the interval $f_+^{-q_{2k+1}+1}(A^+)$, see Section~\ref{feap}); if $k$ is large enough, $\Di (f^{q_{2k+1}}(U),\hat A^+)$ is small (see Section~\ref{distest}).

Let $\gamma>1$ such that $\gamma/\ell^+<1$.
Hence, for any small constant $\varepsilon>0$ and $k$ large enough, there exist two constants $A>0$, $B=\dd(f^{q_{2k+1}}(c^+),c^+)$
such that for any $x\in (f_+^{-q_{2k}-q_{2k+1}}(d^+),c^+)$,

\begin{equation*}
\begin{split}
(1-\varepsilon)(A\dd(x,c^+)^{\ell^+}+B)\leq d(F(x),c^+)\leq (1+\varepsilon)(A\dd(x,c^+)^{\ell^+}+B),\\
(1-\varepsilon)\ell^+A\dd(x,c^+)^{\ell^+-1}\leq | \D F(x)|\leq (1+\varepsilon)\ell^+A\dd(x,c^+)^{\ell^+-1}.
\end{split}
\end{equation*}

Let us consider now the point $y\in [f_+^{-q_{2k}-q_{2k+1}}(d^+),c^+)$ such that
$$(1+\varepsilon)\ell^+A\dd(y,c^+)^{\ell^+-1}=\gamma.$$
By definition of $y$, $|\D F(y)|\leq \gamma$ and from the combinatorics of $F$ on $U$ (see Proposition~\ref{rotation}) $F(y)\not \in [y,c^+]$. This implies that:$$\gamma\geq |\D F(y)|\geq (1-\varepsilon)\ell^+A\dd(y,c^+)^{\ell^+-1}\text{ and }(1+\varepsilon)(A\dd(y,c^+)^{\ell^+}+B)\geq\dd(y,c^+).$$

One gets some constant $C_1>0$ such that
$$\dd(y,c^+)\leq \left(\frac{1}{1+\varepsilon}-\frac{\gamma}{\ell^+(1-\varepsilon)}\right)^{-1}B
\leq C_1. \dd(f^{q_{2k+1}}(c^+),c^+).$$

Let us consider the intervals $I_{k,l}^+$ for $1\leq l \leq a_{2k+2}$ (see Section~\ref{defm0}). We note that $F(I_{k,l}^+)=I_{k,l-1}^+$. Either $I_{k,l}^+$ is contained in $(f_+^{-q_{2k+1}}(d^+),y)$ and:$$|I_{k,l}^+|<\gamma|I_{k,l-1}^+|,$$or $I_{k,l+1}^+$ is contained in $(y,c^+)$
and $|DF|\geq \frac{1-\varepsilon}{1+\varepsilon}\gamma$ on $I_{k,l+1}^+$. One gets:

\begin{equation*}
\begin{split}
\sum_{l=1}^{a_{2k+2}-1}(q_{2k}+lq_{2k+1})|I_{k,l}^+|\leq
&\left(\sum_{n\geq 0}(n+2)\gamma^{-n}\right)q_{2k+1}|I_{k,1}^+|+\\
&\qquad \left(1+\frac{1-\varepsilon}{1+\varepsilon}\gamma\right).q_{2k+2}\dd(f^{q_{2k+1}}(c^+),c^+).
\end{split}
\end{equation*}
\end{proof}

\begin{proposition}\label{primsum} Under \ref{C+}, the primary decomposition $\mathcal{P}_0$ is summable.
\end{proposition}

\begin{proof}[Proof of Proposition \ref{primsum}]
From Graczyk's estimates at Remark~\ref{rkgraczyk}, there are constants $C_2,C_3>0$ such that:
\begin{equation*}
\begin{split}
q_{2k+1}|I_{k,1}^+|&\leq(a_{2k+1}+1)q_{2k}\dd(|I_{k,0}^+|,c^+)\leq\\
&(a_{2k+1}+1)q_{2k}\exp\left(-C_2^{-1}\left(1+\frac{a_{2k+1}}{\ell^+}\right)+C_2\right)|I_{k,0}^+|\leq C_3\;q_{2k}|I_{k,0}^+|.
\end{split}
\end{equation*}
Therefore:
\begin{equation}\label{in2}
q_{2k+1}|I_{k,1}^+|\leq C_3\;q_{2k}|I_{k,0}^+|\leq C_3\;q_{2k}\dd(f^{q_{2k-1}}(c^+),c^+).
\end{equation}
Combining~Lemma \ref{in1} and~(\ref{in2}) we obtain with some uniform constant $C_4>0$,$$\sum_{l=0}^{a_{2k+2}-1}(q_{2k}+lq_{2k+1})|I_{k,l}^+|\leq C_4\left(q_{2k}\dd(f^{q_{2k-1}}(c^+),c^+)+q_{2k+2}\dd(f^{q_{2k+1}}(c^+),c^+)\right).$$
Consequently, \ref{C+} implies that the sum $\sum_{k\geq 0}\sum_{l=0}^{a_{2k+2}-1}N(I_{k,l}^+)|I_{k,l}^+|$ is finite.

Let $T$ be a gap of the primary decomposition. There exists a rough interval $I_{k,l}^+$ adjacent to $T$ such that $T<I_{k,l}^+$. Recall that $N(T)=N(I_{k,l}^+)$ and from Lemma~\ref{ptes}, $|T|<\eta |I_{k,l}^+|$. This ends the proof of the proposition.
\end{proof}

\subsubsection{Step 2. Summability of refined decompositions}\label{sc2}

\begin{proposition}\label{regsum} Under Condition \ref{C+}, any decomposition obtained by refining the primary decomposition $\mathcal{P}_0$ is summable.
\end{proposition}

\begin{proof}[Proof of Proposition \ref{regsum}] As before we consider the rough intervals $I_{k,l}^+$ and denote $I_{k,a_{2k+2}}:=I_{k+1,0}$. Let $T_{k,l}$ be the gap in $\mathcal{P}_0$ between $I_{k,l}^+$ and $I_{k,l+1}^+$.
Any decomposition $\mathcal{P}$ obtained by refining $\mathcal{P}_0$ induces a partition of $I_{k,l}^+\cup T_{k,l}$.

Let us assume first $1\leq l<a_{k,l}$ and an interval $J\subset I_{k,l}^+\cup T_{k,l}$.
By the monotonicity of the order in decompositions (Definition~\ref{dec}.\ref{dec3}), we get,$$N(J)\leq (q_{2k}+(l+1)q_{2k+1})<2(q_{2k}+lq_{2k+1}).$$
Consequently,
\begin{equation}
\begin{split}
\sum_{\substack{J\in\mathcal{P},\\
J\subset I_{k,l}^+\cup T_{k,l}}}N(J)|J|
&\leq 2(q_{2k}+lq_{2k+1})(|I_{k,l}^+|+|T_{k,l}|)\\
&\leq 2(N(I^+_{k,l})|I_{k,l}^+|+N(T_{k,l})|T_{k,l}|).
\end{split}
\end{equation}

We now consider the case $l=0$.
The interval $J\subset I_{k,l}^+\cup T_{k,l}$ contains the point
$z_k:=f_+^{-q_{2k}}(\Psi^{-1}\circ f_+^{-q_{2k}}(c^+))$.
Note that $T:=(z_k, f_+^{-q_{2k}-q_{2k+1}}(c^+))$ is a gap adjacent to $I_{k,1}$.
By Lemma~\ref{ptes}, it has length smaller than $\eta |I_{k,1}^+|$.
Since $f^{2q_{2k}}(z_k)=c^+$, and since the order of a decomposition is monotone,
for any decomposition, the order on $(f_+^{-q_{2k}}(c^+), z_k)$ is smaller than $2q_{2k}$.
This gives
\begin{equation}
\begin{split}
\sum_{\substack{J\in\mathcal{P},\\
J\subset I_{k,0}^+\cup T_{k,1}}}N(J)|J|
&\leq 2q_{2k}(|I_{k,0}^+|+|T_{k,1}|)+\eta(q_{2k}+q_{2k+1})|I_{k,1}^+|\\
&\leq 2(N(I^+_{k,0})|I_{k,0}^+|+N(T_{k,0})|T_{k,0}|) + N(I^+_{k,1})|I_{k,1}^+|.
\end{split}
\end{equation}
One concludes using Proposition~\ref{primsum}:
$$\sum_{J\in\mathcal{P}}N(J)|J|\leq 3 \sum_{J_0\in\mathcal{P}_0}N(J_0)|J_0|<+\infty.$$
\end{proof}

\subsubsection{Step 3. Summability of the measurable partition $\mathcal{M}^+$}

\label{sc3}

\begin{proposition}\label{mesr} Under \ref{C+}, the measurable partition $\mathcal{M}^+$ of $A_L^+$ is summable.
\end{proposition}
\begin{proof}
This is a consequence of Proposition~\ref{regsum} and of the proof of Proposition~\ref{propreg}: we consider again the sequence $(\mathcal{M}_k)$. From Proposition~\ref{regsum}, $\mathcal{M}_0$ is summable. Let $J$ be a gap of $\mathcal{M}_k$ for some $k\in\mathbb{N}$. From Lemma~\ref{distorp1}, we get,
$$\sum_{\substack{J''\in\mathcal{M}_{k+1},\\J''\subset J}}N(J'')|J''|\leq\left( N(J)+\frac{1+D_1}{|A^+|}\sum_{J'\in\mathcal{M}_0}N(J')|J'|\right)|J|$$and $\eta':=\frac{1+D_1}{|A^+|}\sum_{J'\in\mathcal{M}_0}N(J')|J'|$ is arbitrarily small by taking $M_0$ large enough.

One deduces using inequality~(\ref{estgap}),

\begin{equation*}
\begin{split}
\sum_{J''\in\mathcal{M}_{k+1}}N(J'')|J''|&\leq\sum_{J\in\mathcal{M}_k} N(J)|J|+\eta'\sum_{J\text{ gap of }\mathcal{M}_{k}}|J|,\\
&\leq \left(1+\frac{\eta'}{1-\eta(1+ D_1)}\right)\sum_{J\in\mathcal{M}_0} N(J)|J|.
\end{split}
\end{equation*}
We proved that, being $\mathcal{M}_0$ summable, the partition $\mathcal{M}_k$ is summable for any $k$. By the construction of $\mathcal{M}_k$, see proof of Proposition \ref{propreg}, any elent of $\mathcal{M}^+$ can be approximated by elements of $\mathcal{M}_k$. As consequence the partition $\mathcal{M}^+$ of $A_L^+$ is also summable.
\end{proof}

\subsection{Proof of Propositions~\ref{fonda} and~\ref{fondb}} This subsection will be devoted to the proof of Propositions~\ref{fonda} and~\ref{fondb}. Let us start fixing some more definition.

\begin{definition}\label{regf} An interval $J\subset \mathbb{T}^1\setminus\II$ is a \textbf{regular interval for $f$} if there exists an integer $N\geq 1$ such that
\begin{enumerate}

\item\label{regf1} $\forall 0\leq n\leq N-1, \; f^n(J)\cap\II=\emptyset$;

\item\label{regf2} there exists an interval $\hat J\supset J$ (called the extension of $J$) such that $f^N$ in restriction to $\hat J$ is an homeomorphism and one of these cases occurs:

\begin{itemize}
\item $f^N(J)=A^+$ and $f^N(\hat J)=\hat A^+$,
\item $f^N(J)=A^-$ and $f^N(\hat J)=\hat A^-$,
\item $f^N(J)=A$ and $f^N(\hat J)=\hat A$.
\end{itemize}
\end{enumerate}

The integer $N$ is uniquely defined, called the \textbf{order} of $J$ and denoted by $N(J)$.
\end{definition}

\subsubsection{}\label{pr} We now introduce a family $\mathcal{R}$ of regular intervals of $f$ such that:
\begin{enumerate}
\item $\mathcal{R}$ is a measurable partition of $\mathbb{T}^1\setminus\II$;
\item if conditions \ref{C-} and \ref{C+} are satisfied, the order $N:\mathcal{R}\rightarrow\mathbb{N}$ is summable.
\end{enumerate}
It coincides on $A^+_L$ (resp. $A^-_R$) with the restriction of the partition $\mathcal{M}^+$ (resp. $\mathcal{M}^-$), as in Sections~\ref{regp} and~\ref{regm} and on $\mathbb{T}^1\setminus A$ with the partition $\mathcal{S}$ (see Section~\ref{frma}). From Proposition~\ref{frmap} and~\ref{propreg}, we get a measurable partition $\mathcal{R}$ of $\mathbb{T}^1\setminus \II$.

The second property is a direct consequence of Propositions~\ref{frmap} and~\ref{mesr}.

\subsubsection{The family $\mathcal{N}$ is a measurable partition} The main ingredient in the proof of Propositions~\ref{fonda} and~\ref{fondb} is the following:

\begin{proposition}\label{fondap}
\begin{enumerate}
\item\label{fondap1} The family $\mathcal{N}$ is a measurable partition of $\mathbb{T}^1$.
\item\label{fondap2} The order $N^0:\mathcal{N}\rightarrow \mathbb{N}$ is summable if \ref{C-} and \ref{C+} are satisfied.
\end{enumerate}
\end{proposition}

 The strategy of the proof is to refine the measurable partition $\mathcal{R}$ in order to prove that $\mathcal{N}$ is a measurable partition of $\mathbb{T}^1\setminus \II$. For the second claim,  the summability conditions proven in the previous Section for $\mathcal{R}$ will be used. The result will be then extended to $\II$ by pulling back the measurable partition of $\mathbb{T}^1\setminus \II$ by the map $f$ which is a local diffeomorphism on $\interior(\II)$.

We start with introducing refining partitions and we continue proving some thecnical lemmas. 

\subsubsection{}We got from Section~\ref{pr} a measurable partition $\mathcal{Q}_0=\mathcal{R}$ of $\mathbb{T}^1\setminus\II$. By refining inductively $\mathcal{Q}_0$, we build a sequence $(\mathcal{Q}_k)$ of measurable partitions of $\mathbb{T}^1\setminus \II$ which satisfies:

\begin{description}
\item[(P)] For $k\in\mathbb{N}$, any $J\in\mathcal{Q}_k$ either belongs to $\mathcal{N}$ or is a regular interval of $f$.
\end{description}

Let us fix $k\in\mathbb{N}$ and consider $J\in\mathcal{Q}_k$.

\begin{itemize}

\item In the case $J\in\mathcal{N}$, we set $J\in\mathcal{Q}_{k+1}$.

\item In the case $J\not\in\mathcal{N}$, we get $f^{N(J)}(J)\in\{A^-,A^+,A\}$. One refines $J$ by introducing for any $J'\in\mathcal{R}\cup\{\II\}$,$$J''=f^{-N(J)}(J')\cap J.$$
If $J''$ is not empty, we set $J''\in\mathcal{Q}_{k+1}$.
\end{itemize}
Note that since $f^{N(J)}$ is a diffeomorphism and $\mathcal{R}\cup\{\II\}$ a measurable partition of the circle, $\mathcal{Q}_{k+1}$ is a measurable partition of $\mathbb{T}^1\setminus\II$.

\begin{lemma}
 $\mathcal{Q}_{k+1}$ satisfies (P). 
\end{lemma}
\begin{proof}
If $J'=\II$, for any $x\in J''$ we get $N^0(x)=N(J)$ so that $J''\in\mathcal{N}$. If $J'\in\mathcal{R}$, the interval $J'$ is a regular interval and there exists an extension $\hat J'$ which satisfies Definition~\ref{regf}. Let us assume $f^{N(J)}(J)=A^+$ (the cases $f^{N(J)}(J)=A^-$ or $=A$ are similar).

If $J''\not=\emptyset$ then $J'\subset f^{N(J)}(J)=A^+$. In particular, $J''$ contains an interval $\check J$ which is mapped on $\II$
by $f^{N(J)+N(J')}$ and whose $N(J)+N(J')-1$ first iterates are disjoint from $\II$.
One can thus apply the Proposition~\ref{inducp} to $I:=\II$ and to $\hat I:=f^{N(J')}(\hat J')$
(which is one of the intervals $\hat A^+$, $\hat A^-$ or $\hat A$).
Hence, there exists $\hat J''$ which contains $\check J$ and which is mapped homeomorphically on $f^{N(J')}(\hat J')$ by $f^{N(J)+N(J')}$.
In particular, $\hat J''$ contains $J''$, which is mapped on $f^{N(J)}(\hat J)$.
One deduces that $J''$ is regular for $f$ and that $\mathcal{Q}_{k+1}$ satisfies the Property (P).
\end{proof}
\begin{lemma}
There exist some new constants $C>0$ and $\kappa>1$ such that
\begin{equation}
\label{raffin3}
\forall k\in\mathbb{N},\;
\sum_{J''\in\mathcal{Q}_{k+1}\setminus\mathcal{Q}_k}N(J'')|J''|<C.\sum_{J\in\mathcal{Q}_{k}\setminus\mathcal{N}}N(J)|J|
\end{equation}

\begin{equation}
\label{raffin4}
\forall k\in\mathbb{N},\;
\forall J\in\mathcal{Q}_k\setminus\mathcal{N},\;
\sum_{\substack{J''\in\mathcal{Q}_{k+1}\setminus\mathcal{N},\\
J''\subset J}}N(J'')|J''|\leq \kappa^{-2}N(J)|J|+C|J|.
\end{equation}
\end{lemma}
\begin{proof}
 By Proposition~\ref{distorp2}, there exists a constant $D_2>0$ that bounds the distortion of $f^{N(J)}$ on any interval $J\in\mathcal{Q}_k\setminus\mathcal{N}$, for any $k\in\mathbb{N}$.

One deduces for some constant $C>0$,
\begin{equation}
\label{estim1}
\begin{split}
\sum_{J'\in\mathcal{R}}N(J')|f^{-N(J)}(J')\cap J|&\leq\\
D_2&\left(\sum_{J'\in\mathcal{R}}N(J')|J'|\right)\frac{|J|}{\inf(|A^+|,|A^-|,|A|)}\leq C|J|,
\end{split}
\end{equation}
Proposition~\ref{distorp2} implies that $J$ contains a large interval $J''\in\mathcal{Q}_{k+1}\cap\mathcal{N}$:$$|f^{-N(J)}(\II)\cap J|>D_2^{-1}\frac{|\II|}{\sup(|A^+|,|A^-|,|A|)}|J|.$$

Consequently for some constant $\kappa^{-2}:=1-D_2^{-1}\frac{|\II|}{\sup(|A^+|,|A^-|,|A|)}$,
\begin{equation}\label{estim2}
\sum_{J'\in\mathcal{R}}|f^{-N(J)}(J')\cap J|\leq\kappa^{-2}|J|.
\end{equation}

From both estimates~(\ref{estim1}) and~(\ref{estim2}), we get
\begin{equation}\label{lstes}
\begin{split}
\sum_{J'\in\mathcal{R}}
(N(J)&+N(J'))|f^{-N(J)}(J')\cap J|\leq\\
&\leq N(J)\sum_{J'\in\mathcal{R}}|f^{-N(J)}(J')\cap J|+\sum_{J'\in\mathcal{R}}N(J')|f^{-N(J)}(J')\cap J|,\\
&\leq \kappa^{-2} N(J)|J|+C|J|.
\end{split}
\end{equation}

This gives~(\ref{raffin4}). We get also from~(\ref{lstes}),

\begin{equation*}
\begin{split}
N(J)|f^{-N(J)}(\II)\cap J|&+\sum_{J'\in\mathcal{R}}(N(J)+N(J'))|f^{-N(J)}(J')\cap J|\\
&\leq (1+C+\kappa^{-2})N(J)|J|.
\end{split}
\end{equation*}

By summing over $J$, one gets~(\ref{raffin3}):$$\sum_{J''\in\mathcal{Q}_{k+1}\setminus\mathcal{Q}_{k}}N(J'')|J''|\leq (1+C+\kappa^{-2})\sum_{J\in\mathcal{Q}_k\setminus \mathcal{N}}N(J)|J|.$$
\end{proof}

\begin{lemma}\label{sompr1}
Let us consider$$\mathcal{N}'=\bigcup_{l\in\mathbb{N}}\bigcap_{k\geq l}\mathcal{Q}_k.$$
Then $\mathcal{N}'$ is a measurable partition and $\mathcal{N}'=\mathcal{N}\setminus\mathcal{N}^0$.
\end{lemma} 
\begin{proof}
Observe that there exists some constant $\kappa>1$ such that
\begin{equation}
\label{raffin1}
\forall k\in\mathbb{N},\;\forall J\in\mathcal{Q}_k\setminus\mathcal{N},\;
\sum_{\substack{J''\in\mathcal{Q}_{k+1}\setminus \mathcal{N},\\
J''\subset J}}|J''|<\kappa^{-1}|J|.
\end{equation}
The proof of~(\ref{raffin1}) is very similar but easier than the proof of~(\ref{raffin4}). 
One deduces from~(\ref{raffin1}) that for some $C>0$,
\begin{equation}
\label{raffin2}
\forall k\in\mathbb{N},\;
\sum_{J''\in\mathcal{Q}_{k+1}\setminus\mathcal{Q}_k}|J''|<C\kappa^{-k}.
\end{equation}
By construction, $\mathcal{N}'\subset\mathcal{N}\setminus\mathcal{N}^0$ and $\mathcal{Q}_k\setminus\mathcal{N}'=\mathcal{Q}_k\setminus\mathcal{N}$ so that by~(\ref{raffin2}),$$\lambda(\mathbb{T}^1\setminus(\II\cup\bigcup_{J\in\mathcal{N}'}J))=0.$$
\end{proof}

\begin{lemma}\label{sompr} 
 $N^0:\mathcal{N}'\rightarrow \mathbb{N}$ is summable.
\end{lemma}
\begin{proof}
In case \ref{C-} and \ref{C+} are satisfied, $\mathcal{Q}_0=\mathcal{R}$ is summable (Section~\ref{pr}) and 

From~\eqref{raffin3}, each $\mathcal{Q}_k$ is summable. Note that for $k$ large enough and $J\in\mathcal{Q}_k\setminus\mathcal{N}$, the order $N(J)$ is large so that $\kappa^{-2}N(J)+C\leq \kappa^{-1}N(J)$. One deduces from~(\ref{raffin4}) that for some $C'>0$,

\begin{equation*}
\label{raffin5}
\forall k\in\mathbb{N},\;
\sum_{J\in\mathcal{Q}_{k}\setminus\mathcal{N}}
N(J)|J|\leq C'\kappa^{-k}
\sum_{J\in\mathcal{R}}N(J)|J|,
\end{equation*}

and by~(\ref{raffin3}) $N^0:\mathcal{N}'\rightarrow \mathbb{N}$ is summable.
\end{proof}

\subsubsection{\textit{Proof of  Proposition~\ref{fondap}}}\label{conclusion}

 We have proved the proposition on $\mathbb{T}^1\setminus \II$. In restriction to $\interior(\II)$, the map $f$ is a local diffeomorphism. One gets the partition $\mathcal{N}^0$ on $\II$ by pulling back by $f$ the partition $\mathcal{N}\cup \{\II\}$ on $\mathbb{T}^1$. One deduces from Lemma~\ref{sompr1} that $\mathcal{N}^0$ is a measurable partition of $\II$. For almost any $x\in\interior(\II)$, either $f(x)\in\II$ and $N^0(x)=1$ or $f(x)\in J$ for some $J\in\mathcal{N}'$.

The summability of $\mathcal{N}^0$ on compact subsets of $\interior(\II)$ follows from the summability of $\mathcal{N}$ on $\mathbb{T}^1\setminus \II$.
Assumption (A1) gives, on a neighborhood of each critical point, an orientation reversing diffeomorphism $\theta$ which satisfies $f\circ \theta=f$. This shows that the summability of $\mathcal{N}$ near the critical points holds, once it holds on $\mathbb{T}^1\setminus \II$. In case conditions \ref{C-} and \ref{C+} are satisfied, and by Lemma~\ref{sompr}, one concludes that $N^0:\mathcal{N}^0\rightarrow\mathbb{N}$ is summable. The Proposition~\ref{fondap} is now proved.\qed

\subsubsection{\textit{Proof of Proposition~\ref{fonda}.}} From Propositions~\ref{fondap}.(\ref{fondap1}), one knows that the family $\mathcal{N}^0$ is a measurable partition of $\II$ and that for any $J\in\mathcal{N}^0$, the map $T^0$ is a $C^1$ diffeomorphism from $\interior(J)$ onto $\interior(\II)$. The distortion of $T^0$ has been bounded at Proposition~\ref{distorp3}. This implies that $T^0$ is a Markov map of $\interior(\II)$. 

\subsubsection{\textit{Proof of Proposition~\ref{fondb}.}} From Proposition~\ref{fondap}.(\ref{fondap2}), one knows that if conditions \ref{C-} and \ref{C+} are satisfied then $N^0:\mathcal{N}^0\rightarrow \mathbb{N}$ is summable.

Reciprocally let us assume that $N^0:\mathcal{N}^0\rightarrow \mathbb{N}$ is summable. Arguing as in Section~\ref{conclusion}, $\mathcal{N}$ is then summable on a neighborhood of the critical points. The order of any interval in $\mathcal{N}^0$ contained in $[f^{-q_{2k}}_+(c^+),f^{-q_{2k+2}}_+(c^+)]$ is bounded from below by $q_{2k}$. Hence the series $\sum_{k\geq 1}q_{2k}\dd(f^{-q_{2k}}_+(c^+),f^{-q_{2k+2}}_+(c^+))$ is finite.

Note that $f_+^{q_{2k-1}}(c^+)$ belongs to the gap $T$ of the primary decomposition $\mathcal{P}_0$ (see Section~\ref{pint}) adjacent to $I_{k,0}^+$ with $T<I_{k,0}^+$. By Lemma~\ref{ptes}, $|T|\leq \eta |I_{k,0}^+|$. Hence$$d(f_+^{q_{2k-1}}(c^+),f^{-q_{2k}}_+(c^+))\leq \eta \dd(f^{-q_{2k}}_+(c^+),f^{-q_{2k+2}}_+(c^+)).$$

By Remark~\ref{rkgraczyk}, for $k$ large enough,
$$\dd(f^{-q_{2k+2}}_+(c^+),c^+)<\dd(I_{k,0},c^+)<|I_{k,0}|<\dd(f^{-q_{2k}}_+(c^+),f^{-q_{2k+2}}_+(c^+)).$$

Consequently, $\dd(f_+^{q_{2k-1}}(c^+),c^+)$ is bounded by $3\dd(f^{-q_{2k}}_+(c^+),f^{-q_{2k+2}}_+(c^+))$. This shows that \ref{C+} is satisfied. The same holds for \ref{C-}.
\qed

\section{Proof of Theorem \ref{thmA} and Theorem \ref{thmB}}\label{mainthm}

In this final section, as usual, $\phi_\ast\mu$ denotes the push-forward of a Borel measure $\mu$ under a Borel map $\phi$, that is, $\phi_\ast\mu(A)=\mu\big(\phi^{-1}(A)\big)$ for any Borel set $A\subset\cir$. Also, given an interval $J$ we denote by $\chi_J$ its characteristic function, and by $\mu\res{}{J}$ the Borel measure given by $\big(\mu\res{}{J}\big)(A)=\mu(A \cap J)$ for any Borel set $A$.

The key tool in order to build invariant measures that are absolutely continuous with respect to $\lambda$ is the following folklore theorem.

\renewcommand{\theenumi}{\roman{enumi}}

\begin{theorem}[see~\cite{demelovanstrien}, Chapter V, Theorem 2.2]\label{induce} Let $T$ be a Markov map of $\interior(\II)$. Then there exists a $T$-invariant probability measure $\mu^\ast$ which is equivalent to the Lebesgue measure $\lambda$ on $\interior(\II)$. Moreover,
\begin{enumerate}
\item\label{densityinduce} its density $\frac{\dd \mu^\ast}{\dd \lambda}$ is uniformly bounded from above and from below;
\item\label{ergodicinduce} $\mu^\ast$ is ergodic;
\item\label{induce3} the Lyapunov exponent $\int \log|\D T|\dd \mu^\ast$ of $\mu^\ast$ is strictly positive.
\end{enumerate}
\end{theorem}

Assertion~\eqref{induce3} is not stated in~\cite{demelovanstrien} but follows easily from the definitions. Indeed:

\begin{proof}[Proof of Theorem~\ref{induce}, assertion~\eqref{induce3}] We note that for any integer $n\geq 1$, $T^n$ is a Markov map and the constant $D_0$ remains the same. The partition $\mathcal{N}$ has to be refined and replaced by a partition $\mathcal{N}_n$. By bounded distortion (Definition~\ref{markovdef}.\eqref{markovdef2}), if $n$ is large enough, the intervals $J\in\mathcal{N}_n$ can be assumed arbitrarily small so that $$\II>2(1+D_0|\II|)|J|.$$

One deduces that for any $J\in\mathcal{N}_n$ and $x\in J$, $|\D T^n(x)|>2$. Thus one gets $\int \log |\D T^n|\ddi \mu^\ast>\log(2)$. As $\mu^\ast$ is $T$-invariant, one deduces$$\int \log|\D T| \ddi \mu^\ast>\frac{\log(2)}{n}>0.$$
\end{proof}

\subsection{Proof of Theorem \ref{thmA}}\label{pfs1} From Proposition~\ref{fonda} one knows that $T^0$ is a Markov map of $\interior(\II)$, according to Definition \ref{markovdef}. By Theorem~\ref{induce} (folklore theorem) $T^0$ preserves a probability measure $\mu^\ast$ in $\interior(\II)$ which is absolutely continuous with respect to the Lebesgue measure. We define the following Borel measure in the unit circle:$$\bar \mu=\sum_{J\in\mathcal{N}^0}\sum_{k=0}^{N(J)-1}f_\ast^k\big(\mu^\ast\res{}{J}\big)\,.$$

Since $\sum_{J\in\mathcal{N}^0}f_\ast^{N(J)}\big(\mu^\ast\res{}{J}\big)=T_\ast^0\mu^\ast=\mu^\ast=\sum_{J\in\mathcal{N}^0}\mu^\ast\res{}{J}$ one obtains $f_\ast\bar\mu=\bar\mu$, that is, the measure $\bar\mu$ is $f$-invariant. It is also clear that $\bar\mu$ is absolutely continuous with respect to the Lebesgue measure (since $\mu^\ast$ itself is absolutely continuous and $f$ is smooth). The restrictions of $\bar \mu$ and $\lambda$ to $\II$ are equivalent. The same holds on $\mathbb{T}^1$, from the invariance and the following lemma:
\begin{lemma}\label{l.recouvre}
There exists $N\geq 1$ such that $f^N(\II)=\mathbb{T}^1$.
\end{lemma}

\begin{proof}[Proof of Lemma \ref{l.recouvre}] Since the maps $\tilde f_-\leq \tilde f_+$ are distinct, the irrational rotation numbers $\rho^-,\rho^+$ are distinct.
Since the orbits of $\tilde c^+$ (resp. $\tilde c^-$) have rotation number $\rho^+$ (resp. $\rho^-$),
for $n\geq 1 $ large enough the image $f^n([\tilde c^+,\tilde c^-])$ has length larger than $1$.
\end{proof}

Therefore, as pointed out in the introduction, any map in $\Bimod$ preserves a $\sigma$-finite measure which is equivalent to the Lebesgue measure. The proof of Theorem \ref{thmA} has now three steps.

\begin{enumerate}
\item \label{pf1} We claim first that if the summability conditions \ref{C-} and \ref{C+} are satisfied, $\bar \mu$ is finite. Indeed, note first that:
\begin{align*}
\bar\mu(\cir)&=\sum_{J\in\mathcal{N}^0}\sum_{k=0}^{N(J)-1}f_\ast^k\big(\mu^\ast\res{}{J}\big)(\cir)=\sum_{J\in\mathcal{N}^0}\!N(J)\,\mu^\ast(J)\,.
\end{align*}

Therefore, to prove that $\bar\mu$ is a finite measure we need to prove that the series $\displaystyle\sum_{J\in\mathcal{N}^0}\!N(J)\,\mu^\ast(J)$ is finite. By Proposition~\ref{fondb}, conditions \ref{C-} and \ref{C+} are satisfied if and only if $\sum_{J\in\mathcal{N}^0}N(J)|J|$ is finite. By Assertion \eqref{densityinduce} of Theorem~\ref{induce} the density $\frac{\dd \mu^\ast}{\dd \lambda}$ is uniformly bounded from above, and then the sequence $\big\{\mu^\ast(J)/|J|\big\}_{J\in\mathcal{N}^0}$ is bounded. This implies that $\displaystyle\sum_{J\in\mathcal{N}^0}\!N(J)\,\mu^\ast(J)$ is finite, as claimed.

Therefore, after normalization, $f$ preserves a \emph{probability} measure $\mu$ which is absolutely continuous with respect to the Lebesgue measure. Namely:$$\mu=\left(\frac{1}{\sum_{J\in\mathcal{N}^0}\!N(J)\,\mu^\ast(J)}\right)\sum_{J\in\mathcal{N}^0}\sum_{k=0}^{N(J)-1}f_\ast^k\big(\mu^\ast\res{}{J}\big)\,.$$

Since the system $(T^0,\mu^\ast)$ is ergodic, the same is true for $(f,\mu)$. Finally, note that:
\begin{align*}
\int\log|\D T^0|\,d\mu^\ast&=\sum_{J\in\mathcal{N}^0}\int\log|\D f^{N(J)}|\,d(\mu^\ast\res{}{J})\\
&=\sum_{J\in\mathcal{N}^0}\sum_{k=0}^{N(J)-1}\int\log|\D f| \circ f^{k}\,d(\mu^\ast\res{}{J})\\
&=\sum_{J\in\mathcal{N}^0}\sum_{k=0}^{N(J)-1}\int\log|\D f| \,d\,f_\ast^k(\mu^\ast\res{}{J})=\int\log|\D f|\,d\bar\mu\,.
\end{align*}

From Assertion \eqref{induce3} of Theorem~\ref{induce} we know that $(T^0,\mu^\ast)$ has a Lyapunov exponent which is strictly positive, and therefore this is also true for $(f,\bar\mu)$, and then for $(f,\mu)$.

\item\label{pf2} Reciprocally, we show in the next section that if $f$ preserves a probability measure $\nu$ which is absolutely continuous with respect to the Lebesgue measure, the summability conditions \ref{C-} and \ref{C+} are satisfied.

\item\label{pf3} One ends by proving Theorem \ref{thmA} as a direct consequence of Steps~\eqref{pf1} and~\eqref{pf2}: let us assume that $f$ preserves a probability measure $\nu$ which is absolutely continuous with respect to the Lebesgue measure $\lambda$. From Step~\eqref{pf2}, one gets the summability conditions \ref{C-} and \ref{C+}. From Step~\eqref{pf1}, $f$ preserves a probability measure $\mu$ which is ergodic and equivalent to the Lebesgue measure. In particular $\nu\ll\mu$ and then $\nu=\mu$ by the ergodicity of $\mu$. Again from Step~\eqref{pf1} we already know that $\mu$ has a strictly positive Lyapunov exponent, and this implies that $\mu$ has positive metric entropy (see for instance \cite{ledra}). Finally, the fact that $\mu$ is a global physical measure for $f$ (see Definition \ref{deffis}) follows from Birkhoff's Ergodic Theorem, since $\mu$ is ergodic and equivalent to the Lebesgue measure.
\end{enumerate}

It remains to prove Step~\eqref{pf2}, that is, the fact that both conditions \ref{C-} and \ref{C+} are \emph{necessary} for an element in $\Bimod$ to leave invariant an absolutely continuous \emph{probability} measure (recall that, as we saw in Section \ref{pfs1}, any map in $\Bimod$ preserves a $\sigma$-finite measure which is equivalent to Lebesgue).

\subsection{Existence $\implies$ summability conditions \ref{C-} and \ref{C+}} Let $\nu$ be an $f$-invariant Borel probability measure in $\cir$, which is absolutely continuous with respect to Lebesgue.

\subsubsection{}\label{s1} We claim first that $\nu(\II)>0$. Indeed, from Proposition~\ref{fondap} we know that Lebesgue-almost every point returns to $\II$, that is:
\begin{equation}\label{aeret}
\lambda\big(\big\{x\in\mathbb{T}^1:f^n(x)\in\cir\!\setminus\!\II\,\mbox{ for all $n \geq 1$}\big\}\big)=0\,.
\end{equation}

Since $\nu$ is absolutely continuous with respect to Lebesgue, we get from \eqref{aeret} that $\nu\big(\cup_{n \geq 1}f^{-n}(\II)\big)=1$, that is, $\nu$-almost every point returns to $\II$. This implies at once that $\nu(\II)>0$, since $\nu$ is $f$-invariant.

We consider now the Birkhoff averages for the characteristic function $\chi_\II$ of $\II$ under the action of $f$, that is, for any $n\in\mathbb{N}$ and $x\in\mathbb{T}^1$, we denote:$$S_n(x)=\frac{1}{n}\sum_{k=0}^{n-1}\chi_\II(f^k(x)).$$

Let $A_0=\{x\in \II,\; \displaystyle\lim_{n\to+\infty}S_n(x)>0\}$. We claim that $\lambda(A_0)>0$.

Indeed, by Birkhoff's Ergodic Theorem the sequence $S_n(x)$ converges for $\nu$-almost every point $x$ to some non-negative real number $h(x)\geq 0$. The measurable function $h:\mathbb{T}^1\rightarrow \mathbb{R}$ is $f$-invariant and has a strictly positive integral under $\nu$:$$\int_{\mathbb{T}^1} h\ddi \nu=\int_{\mathbb{T}^1}\chi_\II\ddi\nu=\nu(\II)>0\,.$$
In particular $h(x)>0$ for $x$ in a subset of $\mathbb{T}^1$ with positive $\nu$-measure. Since $h$ and $\nu$ are invariant, and using~\eqref{aeret}, one gets $\nu(A_0)>0$. This proves the claim, since $\nu$ is absolutely continuous with respect to Lebesgue.

\subsubsection{}\label{s2} Let us consider now the Markov system $(T^0,\mu^\ast)$ and the Birkhoff averages of the return time $N^0$: for any $m\in\mathbb{N}$ and $\lambda$-almost any $x\in\II$, we denote:$$R_m(x)=\frac{1}{m}\sum_{k=0}^{m-1}N^0\big((T^0)^k(x)\big)\,.$$

\begin{lemma}\label{lemnec} There exists a Borel set $A\subset\II$ with $\mu^\ast(A)>0$ such that for all $x \in A$ we have $\displaystyle\lim_{m\to+\infty}R_m(x)<\infty$.
\end{lemma}

\begin{proof}[Proof of Lemma \ref{lemnec}] From Theorem \ref{induce} we know that $\mu^\ast$ is ergodic under the action of $T^0$, and that it is equivalent to Lebesgue. Birkhoff's Ergodic Theorem implies then that there exists a Borel set $A_1\subset\II$ such that $\lambda(\II \setminus A_1)=0$ and such that for all $x \in A_1$ the sequence $R_m(x)$ converges as $m$ goes to infinity.

Let $A= A_0\cap A_1$, where $A_0\subset\II$ was obtained in \ref{s1}, and note that $\mu^\ast(A)>0$ since $\lambda(A)>0$. We fix $x \in A$ and for each $m\in\nt$ let $n=n(m)\in\nt$ be given by $n=m\,R_m(x)$, that is:$$n(m)=\sum_{k=0}^{m-1}N^0\big((T^0)^k(x)\big).$$

By definition of $T^0$ and $N^0$, we know that between iterates $0$ to $n-1$ the $f$-orbit of $x$ falls precisely $m$ times in $\II$, that is, $m=\sum_{k=0}^{n-1}\chi_\II(f^k(x))=n\,S_n(x)$. In particular:
\begin{equation}\label{claims2}
R_m(x)\,S_{n(m)}(x)=1\quad\mbox{for all $m \geq 1$.}
\end{equation}

Since $x \in A \subset A_0$ we know from \ref{s1} that $\displaystyle\lim_{n\to+\infty}S_n(x)>0$, and then we obtain from \eqref{claims2} that $\displaystyle\lim_{m\to+\infty}R_m(x)<\infty$.
\end{proof}

\subsubsection{} From Birkhoff's Ergodic Theorem we know that for $\mu^\ast$ almost every $x\in\II$ the sequence $R_m(x)$ converges to $\int_\II N^0\ddi \mu^\ast$ as $m$ goes to infinity, and therefore we obtain from Lemma \ref{lemnec} the finiteness condition $\int_{\II}N^0\,d\mu^{\ast}<\infty$.

As remarked at Section~\ref{pfs1}, Assertion \eqref{densityinduce} of Theorem~\ref{induce} implies that this finiteness condition is equivalent to the fact that the family $\mathcal{N}^0$ is summable. By Proposition \ref{fondb}, the latter is equivalent to the fact that both conditions \ref{C-} and \ref{C+} hold.
This concludes the proof of Step~\eqref{pf2} (and the proof of Theorem \ref{thmA}).

\subsection{Proof of Theorem \ref{thmB}}\label{provadoB}
\subsubsection{} From Graczyk's estimates already mentioned in Section \ref{S:geomest}, it follows that $\dd\big(f_+^{q_{2k-1}^+}(c^+),c^+\big)$ goes to zero \emph{super-exponentially fast} \cite[Definition 1.1 and Theorem 1]{jacekmanuscript}, and from this fact we will deduce in this section that both \ref{C-} and \ref{C+} hold under a suitable Diophantine condition.

The proof goes by elementary calculus: let $l=\max\big\{\ell^-\,,\,\ell^+\big\}>1$ and fix some constant $\beta=\beta(l)$ determined by$$0<\beta<\sqrt{1+1/2l}-1\,.$$

Note that:$$1<(1+\beta)^2<1+1/2l=\left(\frac{l+1/2}{l+1}\right)\left(1+\frac{1}{l}\right)<1+1/l\,.$$

We assume from now on that both irrational numbers $\rho^-$ and $\rho^+$ are Diophantine with exponent $\beta$. Using Corollary~\ref{bound-q-n}, this implies that $q_{n+1}^+\leq cte. (q_n^+)^\beta$ for each $n$. Hence, there exists $C=C(\rho^-,\rho^+)>0$ such that for any $n\in\mathbb{N}$ we have$$\log(q_{n}^+)\leq C(1+\beta)^n\quad\mbox{and}\quad\log(q_{n}^-)\leq C(1+\beta)^n\,.$$

From \cite[First Basic Lemma, page 271]{jacekmanuscript} there exists a constant $\gamma\in(0,1)$ such that for all $k\in\nt$ we have:
\begin{align*}
\log\dd\big(f_+^{q_{2k-1}^+}(c^+),c^+\big)&\leq\log\gamma\!\prod_{i=1}^{k-1}\left(1+\frac{a_{2i+1}^+}{\ell^+}\right)\leq\log\gamma\,(1+1/\ell^+)^{k-1}<0\,.
\end{align*}

We have used that $a_{2i+1}^{+} \geq 1$ for all $i\in\nt$. With this at hand we have:
\begin{align*}
\log\big(q_{2k}^+\dd(f^{q_{2k-1}^+}(c^+),c^+)\big)&\leq C(1+\beta)^{2k}+\log\gamma\,(1+1/\ell^+)^{k-1}\\
&\leq \left[C\left(\frac{\ell^++1/2}{\ell^++1}\right)^{k}+\frac{\log\gamma}{1+1/\ell^+}\right](1+1/\ell^+)^{k}\,.
\end{align*}

Let $k_0\in\nt$ be such that:$$0<\left(\frac{\ell^++1/2}{\ell^++1}\right)^{k}<\frac{-\log\gamma}{2C}\left(\frac{\ell^+-1}{\ell^++1}\right)\quad\mbox{for all $k \geq k_0$\,.}$$

Then:$$C\left(\frac{\ell^++1/2}{\ell^++1}\right)^{k}+\frac{\log\gamma}{1+1/\ell^+}<\frac{\log\gamma}{2}<0\quad\mbox{for all $k \geq k_0$\,.}$$

With this at hand we obtain that:$$\sum_{k\geq k_0}q^+_{2k}\dd\big(f^{q^+_{2k-1}}(c^+),c^+\big)\leq\sum_{k\geq k_0}\exp\left(\frac{\log\gamma}{2}\,(1+1/\ell^+)^{k}\right),$$which implies condition \ref{C+} since the ratio:$$\frac{\exp\left(\frac{\log\gamma}{2}\,(1+1/\ell^+)^{k+1}\right)}{\exp\left(\frac{\log\gamma}{2}\,(1+1/\ell^+)^{k}\right)}=\exp\left(\frac{\log\gamma}{2\ell^+}\,(1+1/\ell^+)^{k}\right)$$belongs to $(0,1)$ and goes to zero as $k$ goes to infinity. Note, finally, that the same arguments hold for \ref{C-} and $f_-$ (the setting is symmetric).

\subsubsection{} Let us consider the opposite situation. From \cite[Second Basic Lemma, page 271]{jacekmanuscript} there exists a constant $\gamma\in(0,1)$ such that for all $k\in\nt$ we have:
\begin{align*}
\log\dd\big(f_+^{q_{2k-1}^+}(c^+),c^+\big)&\geq\log\gamma\!\prod_{i=1}^{k-1}\left(1+\frac{\ell^++1}{\ell^+-1}a_{2i+1}^+\right).
\end{align*}

As an example, consider an irrational number $\rho^+$ such that $a_{2k}^+\geq\exp(k^k)$ and $a_{2k+1}^+=1$ for any integer $k$. Then we have:
\begin{align*}
\log\big(q_{2k}^+\dd(f^{q_{2k-1}^+}(c^+),c^+)\big)&\geq\log(q_{2k}^+)+\log\gamma\prod_{i=1}^{k-1}\left(1+\frac{\ell^++1}{\ell^+-1}a_{2i+1}^+\right)\\
&\geq k^k+\log\gamma\left(1+\frac{\ell^++1}{l^+-1}\right)^{k-1}\,.
\end{align*}

This shows that Condition \ref{C+} fails.

\appendix

\section{Topological transitivity}\label{apptopasp} In this appendix we show that maps in the class $\Bimod$ (Definition \ref{ourspace} in Section \ref{setting}) are topologically mixing in the whole circle. This is done without using conditions \ref{C-} and \ref{C+}.

\begin{proposition}\label{topmix} For any endomorphism $f\in\Bimod$ and any non-trivial interval $I$, there exists $N\geq 1$ such that $f^N(I)=\mathbb{T}^1$. In particular $f$ is topologically mixing.
\end{proposition}

\begin{proof}[Proof of Proposition \ref{topmix}] Let $I$ be a non-trivial interval. By Proposition~\ref{fondap}, there exists a positive iterate $f^m(I)$ which intersects $\interior(\II)$.

Let us fix some $\varepsilon>0$ smaller than $\big|f^m(I)\cap\interior(\II)\big|$. The definition of Markov maps implies that for $\ell\geq 1$ large enough, there exists a dense set of intervals
$J$ in $\II$ such that $(T^0)^\ell(J)=\II$. Consequently $f^{m+\ell}(J)\supset \mathbb{T}^1$.
Then Lemma~\ref{l.recouvre} concludes.
\end{proof}

\section{Ma\~n\'e's theorem for induced maps}\label{appMa}

We prove here our version of Ma\~n\'e's theorem (Theorem~\ref{maneinduce}) for induced maps. The main steps of the proof are the same as in~\cite{demelovanstrien}, Chapter III.5 and some classical parts of the proof are only sketched here.

\subsection{} \textbf{Step 0: control of the distortion.} Let $I$, $T(I),\cdots,T^{n-1}(I)$ be a sequence of intervals in $\mathcal{J}$. Then, for any $x,y\in I$,$$\frac{|\D T^n(x)|}{|\D T^n(y)|}\leq \prod_{k=1}^{n}(1+D_d|T^k(I)|)\leq \exp\left(D_d\sum_{k=1}^n|T^k(I)|\right).$$

\subsection{} \textbf{Step 1: no wandering interval.} Let us consider for any $n\in\mathbb{N}$ the supremum $e_n$ of $|I|$ over all connected components $I$ of $\cap_{k=0}^{n-1}T^{-k}(\mathcal{J})$. We claim that

\begin{equation*}
\begin{CD}
e_n@>>{n\rightarrow\infty}>0.
\end{CD}
\end{equation*}

One proves this claim by contradiction and assumes that there exists some interval $I$ with non-empty interior such that $T^n(I)\subset \mathcal{J}$ for any $n\in\mathbb{N}$. One can suppose that $I$ is maximal with respect to the inclusion for this property. Note that for any interval $U\subset \mathcal{J}$, $\closure(U)$ is contained in $\mathcal{J}$. Thus, $I$ is also compact. The sequence $(T^n(I))$ is not preperiodic, otherwise for some large $n$, $T^n(I)$ would contain a periodic point which is not both hyperbolic and repulsive. The Markov properties of $T$ imply then that all the intervals $T^n(I)$ are disjoint.

Let us fix $\eta_0=(2(1+D_d))^{-1}$. Since any $J\in\mathcal{N}^0$ is a compact subinterval of the open interval $(0,1)$, there exists $0<\eta_1<\eta_0$ such that for any $J\in \mathcal{N}^0$,

\begin{itemize}
\item either $J\subset (0,\eta_0)\cup (1-\eta_0,1)$,
\item or $J\subset (\eta_1,1-\eta_1)$.
\end{itemize}

For $J\in \mathcal{N}^0$ contained in $(0,\eta_0)\cup(1-\eta_0,1)$, the derivative of $T$ on $J$ is greater than $2$. If for any $n$ larger than some integer $n_0$, the interval $J\in\mathcal{N}^0$ that contains $T^n(I)$ is included in $(0,\eta_0)\cup(1-\eta_0,1)$, the derivative of $T^k$ on $T^{n_0}(I)$ is larger than $2^k$ for any $k\in\mathbb{N}$. This implies $|T^{n_0}(I)|=0$ and this is a contradiction.

Consequently, there exists an infinite sequence $n_1<n_2<\cdots$ such that for any $i\in\mathbb{N}$ the interval $J_{n_i}\in\mathcal{N}^0$ containing $T^{n_i}(I)$ is included in $(\eta_1,1-\eta_1)$. Since the intervals $T^n(I)$ are disjoint, on $I$ the distortion of $T^k$ for any $k\in\mathbb{N}$ is bounded by $\exp(D_d)$ (Step 0). For $\varepsilon>0$ and any $i$, we introduce the interval $U_{n_i,\varepsilon}$ containing $T^{n_i}(I)$ such that both components of $U_{n_i,\varepsilon}\setminus T^{n_i}(I)$ have length $\varepsilon|T^{n_i}(I)|$. If $\varepsilon>0$ is small enough, $U_{n_i,\varepsilon}$ is included in $[0,1]$.

Let $I(n_i,\varepsilon)$ be the interval containing $I$ that is mapped onto $U_{n_i,\varepsilon}$ by $T^{n_i}$. By the same argument as in Proposition~\ref{distorp1}, if $\varepsilon>0$ is small, the distortion of $T^{n_i}$ on $I(n_i,\varepsilon)$ is bounded by $2(1+\exp(D_d))$. Consequently,$$|I(n_i,\varepsilon)|>\left(1+(1+\exp D_d)^{-1})\varepsilon\right)|I|.$$

One deduces that the intersection$$I'=\bigcap_i I(n_i,\varepsilon)$$is an interval larger than $I$ and that all its iterates are contained in $\mathcal{J}$. This contradicts the maximality of $I$.

\subsection{} \textbf{Step 2: growth of derivatives, the periodic case.}

\begin{lemma}\label{lgdp} There exists a constant $C_d>0$ such that for any periodic orbit $x_0,\cdots,x_q=x_0$ of $T$ with minimal period $q$,$$\left|\D T^q(x_0)\right|\geq \frac{C_d}{e_q}.$$
\end{lemma}

\begin{proof} Let $x_0,\cdots,x_q=x_0$ be any periodic orbit $\mathcal{O}$ of $T$ in $\mathcal{J}$ with minimal period $q\geq 1$ (one will also note $x_k=x_{k+q}$ for $-q\leq k<0$). Recall that we have introduced at Step 1 some constants $\eta_0$ and $\eta_1$. For $J\in \mathcal{N}^0$ contained in $(0,\eta_0)\cup(1-\eta_0,1)$, the derivative of $T$ on $J$ is greater than $2$. Hence, if one assumes that $\mathcal{O}\cap (\eta_1,1-\eta_1)=\emptyset$, one gets$$|\D T^q(x_0)|\geq2^q.$$

From now on, one consider the opposite case and one can suppose that
\begin{itemize}
\item $x_0$ belongs to some interval $J_0\in\mathcal{N}^0$ contained in $(\eta_1,1-\eta_1)$;
\item $((0,x_0)\cap J_0)\cap \mathcal{O}=\emptyset.$ (i.e. $x_0$ is the closest point in $\mathcal{O}\cap J_0$ to $0$).
\end{itemize}

Let $0<s\leq q$ be the smallest integer such that $x_s$ belongs to $(\eta_1,1-\eta_1)$. There exists an interval $I_s\in\{[0,x_s],[x_s,1]\}$ and an interval $I_0\subset J_0$ such that

\begin{itemize}
\item $T^s$ maps $I_0$ homeomorphically onto $I_s$;
\item $I_0=[x_0-\xi,x_0]$ for some real number $\xi>0$.
\end{itemize}

One defines for any $s-q\leq k\leq s$ the interval $I_k$ as the unique interval mapped homeomorphically onto $I_s$ by $T^{s-k}$ which contains $x_k$ in its boundary.

On $I_k$, for $0<k<s$, the derivative of $T$ is larger than $2$ so that$$\sum_{k=1}^{s}|I_k|\leq \sum_{n\geq 0}2^{-n}=2.$$

The intervals $I_k$ for $s-q<k\leq 0$ are disjoint: if one supposes $I_k\cap I_l\not =\emptyset$ for $s-q<k<l\leq 0$ then $I_{k-l}\cap I_0\not =\emptyset$. Recall that $I_{k-l}\subset J_0$ and that $x_{k-l}\not\in (0,x_0)\cap J_0$. This implies $x_0\in I_{k-l}$. Hence, $x_{l-k}\in I_0$ which is a contradiction. Consequently,$$\sum_{k=s-q+1}^{s}|I_k|\leq\sum_{k=s-q+1}^{0}|I_k|+\sum_{k=1}^{s}|I_k|\leq 3.$$

This implies by Step 0 that the distortion of $T^q$ on $I_{s-q}$ is bounded:$$\left|\D T^q(x_0)\right|\geq \exp(-3 D_d)\frac{|I_s|}{|I_{s-q}|}\geq \exp(-3 D_d)\frac{\eta_1}{e_q}.$$
\end{proof}

\begin{corollary}\label{gper} There exists $n_0\geq 1$ and a constant $\kappa_0>1$ such that for any periodic orbit $x_0,\cdots,x_q=x_0$ of $T$ and any interval $I$ containing $x_0$ such that $I$, $T(I),\cdots,T^{q+n_0}(I)$ are contained in $\mathcal{J}$ and $I$, $T(I),\cdots,T^{q}(I)$ are disjoint, then, $$\forall x\in I, \; \left|\D T^q(x)\right|\geq \kappa_0.$$
\end{corollary}

\begin{proof} There exists a constant $\kappa_0>1$, such that for any periodic orbit $x_0$,\dots, $x_q=x_0$ with minimal period $q$, $|\D T^q(x_0)|>\kappa_0^2$: this is obvious if $q$ is large by Lemma~\ref{lgdp}; for small values of $q$, if $x_0$ belongs to a small interval $J\in\mathcal{N}^0$, this is true again by control of the distortion. By hyperbolicity, only a finite number of periodic orbits do not fall in one of the two previous cases. Hence, the modulus of their derivatives is bounded from below by finiteness.

Let $I$ be an interval such that $I$, $T(I),\cdots T^{q-1}(I)$ are disjoint and contained in $\mathcal{J}$. If $|D T^q(x_0)|$ is large enough, $|\D T^q|$ is large again on $I$ by Step 0.

There are only a finite number of periodic orbits such that $|D T^q(x_0)|$ is not large. In this case, assuming that $n_0$ is large enough and that $I$, $T(I),\cdots, T^{q-1}(I)$ are contained in $\mathcal{J}$, then $|I|<e_{n_0}$ is small (Step 1) so that by continuity, $|\D T^q|$ is greater than $\kappa_0$ on $I$.
\end{proof}

\subsection{} \textbf{Step 3: growth of derivatives, the general case.} We now prove the theorem. Let $x_0,\cdots, x_n$ be an orbit of $T$ in $\mathcal{J}$ and for any $0\leq k\leq n$ let $I_k$ be the compact interval which contains $x_k$ and is sent homeomorphically onto $[0,1]$ by $T^{n-k}$. One claims that there is a uniform constant $L$ that bounds $\sum_{0\leq k\leq n} |I_k|$. One deduces by Step 0 that for any $x$, $y$ in $I_0$,$$\left|\frac{\D T^n(x)}{\D T^n(y)}\right|\leq \exp(K_dL).$$

This implies$$|\D T^n(x_0)|\geq \frac{\exp(K_dL)}{e_n},$$which concludes the proof with Step 1. It remains to prove the claim. One considers the intervals $I_k$ for $0\leq k\leq n-n_0$. Note that for any $0\leq k\leq l\leq n-n_0$ either $I_k$ and $I_l$ are disjoint or $I_k\subset I_l$. Consequently, there exists some indices $0\leq i_1<i_2<\cdots<i_s\leq n-n_0$ such that any two distinct intervals $I_{i_{m}}$ and $I_{i_{m'}}$ are disjoint and any interval $I_k$ for $0\leq k\leq n-n_0$ is contained in some $I_{i_m}$.

Let us consider $I_k,I_l\subset I_{i_m}$ with $0\leq k<l\leq i_m$. One shows first that $\kappa_0|I_k|\leq |I_l|$ with the constant $\kappa_0$ of Corollary~\ref{gper}. Note that it is sufficient to assume that $l$ is minimal for the inclusion with those properties. There exists a point $z\in I_k$ whose orbit is periodic with minimal period $l-k$ and an interval $I\subset I_{i_m}$ that is sent homeomorphically onto $I_{i_m}$ by $T$ and contains $I_k$. By the minimality of $l$, the iterates $I,\cdots,T^{l-k-1}(I)$ are disjoint. Hence, using Corollary~\ref{gper}, $\kappa_0|I_k|\leq |I_l|$. One deduces immediately that $\sum_{0\leq k\leq n}|I_k|$ is bounded by $$L=\frac{\kappa_0}{\kappa_0-1}+n_0.$$

\end{document}